\documentclass[11pt,  reqno]{amsart}

\usepackage{amsmath,amssymb,amscd,amsthm,amsxtra, esint}

\usepackage[english,french]{babel}
\usepackage[implicit=true]{hyperref}
\usepackage{color}

\allowdisplaybreaks[2]

\usepackage{bbm}

\usepackage{verbatim}
\DeclareFontFamily{U}{mathx}{\hyphenchar\font45}
\DeclareFontShape{U}{mathx}{m}{n}{
      <5> <6> <7> <8> <9> <10>
      <10.95> <12> <14.4> <17.28> <20.74> <24.88>
      mathx10
      }{}
\DeclareSymbolFont{mathx}{U}{mathx}{m}{n}
\DeclareFontSubstitution{U}{mathx}{m}{n}
\DeclareMathAccent{\widecheck}{0}{mathx}{"71}
\DeclareMathAccent{\wideparen}{0}{mathx}{"75}

\newtheorem{theorem}{Theorem} [section]

\newtheorem{lemma}[theorem]{Lemma}
\newtheorem{proposition}[theorem]{Proposition}
\newtheorem{corollary}[theorem]{Corollary}

\theoremstyle{definition}
\newtheorem{remark}[theorem]{Remark}

\DeclareMathOperator*{\supp}{supp}

\newcommand{\noi}{\noindent}
\newcommand{\Z}{\mathbb{Z}}
\newcommand{\R}{\mathbb{R}}
\newcommand{\C}{\mathbb{C}}
\newcommand{\T}{\mathbb{T}}

\newcommand{\Psia}{\Psi}
\newcommand{\Psim}{\Psi}

\let\Im=\undefined\DeclareMathOperator*{\Im}{Im}

\let\P= \undefined
\newcommand{\P}{\mathbb{P}}

\renewcommand{\L}{\mathcal{L}}

\newcommand{\N}{\mathcal{N}}

\newcommand{\F}{\mathcal{F}}

\def\norm#1{\left\|#1\right\|}

\newcommand{\nb}{\nabla}

\newcommand{\Dl}{\Delta}
\newcommand{\eps}{\varepsilon}

\newcommand{\ft}{\widehat}

\newcommand{\cj}{\overline}

\renewcommand{\l}{\ell}

\newcommand{\les}{\lesssim}

\newcommand{\jb}[1]
{\langle #1 \rangle}

\newcommand{\ind}{\mathbf 1}

\newcommand{\Irm}{\mathrm{I}}
\newcommand{\IIrm}{\mathrm{I\!I}}

\numberwithin{equation}{section}
\numberwithin{theorem}{section}

\newcommand{\NN}{\mathbb{N}}
\newcommand{\EE}{\mathbb{E}}
\newcommand{\one}{\mathbbm{1}}
\newcommand{\sub}[1]{_{#1}}
\newcommand{\longeqn}{\hspace{2em}&\hspace{-2em}}

\newcommand{\xsb}[3]{\left\lVert#1\right\rVert_{X^{#2,#3}}}
\newcommand{\xsbt}[4]{\left\lVert#1\right\rVert_{X^{#2,#3}({#4})}}
\newcommand{\intud}[4]{\int_{#1}^{#2}\,{#3}\, d{#4}}
\newcommand{\intd}[3]{\int_{#1}{#2}\, d{#3}}
\newcommand{\brac}[1]{\left(#1\right)}
\newcommand{\supT}{\sup_{t\in[0,T]}}

\begin{document}
\selectlanguage{english}

\title[SNLS on $\T^d$]
{Stochastic nonlinear Schr\"{o}dinger equations on tori}

\author{Kelvin Cheung}
\address{
Maxwell Institute for Mathematical Sciences\\
Department of Mathematics\\
Heriot-Watt University\\
Edinburgh\\
UK\\
EH14 4AS}
\email{K.K.Cheung-3@sms.ed.ac.uk}

\author{Razvan Mosincat}
\address{
Maxwell Institute for Mathematical Sciences\\
School of Mathematics\\
University of Edinburgh\\ 
Edinburgh\\
UK\\
EH9 3FD}
\email{r.o.mosincat@sms.ed.ac.uk}



\date{\today}

\begin{abstract}
We consider the stochastic nonlinear Schr\"odinger equations (SNLS) posed on
 $d$-dimensional tori with either
additive or multiplicative stochastic forcing. In particular, for
the one-dimensional cubic SNLS, we prove global well-posedness in
$L^2(\T)$.  As for other power-type nonlinearities, namely
(i) (super)quintic when $d = 1$ and (ii) (super)cubic when $d \geq 2$, we prove
local well-posedness in all scaling-subcritical Sobolev spaces and
global well-posedness in the energy space for the defocusing,
energy-subcritical problems.
\end{abstract}

\maketitle

\section{Introduction}

\subsection{Stochastic nonlinear Schr\"{o}dinger equations} 
In this paper, we study the following Cauchy problem 
associated to a stochastic nonlinear Schr\"{o}dinger equation (SNLS) of the form: 
\begin{equation}
\label{SNLS}
\begin{cases}
i\partial_t u - \Delta u  \pm |u|^{2k}u = F(u,\phi\xi)\\
u|_{t=0}=u_0 \in H^s(\T^d)
\end{cases}
\ (t,x)\in (0,\infty)\times\T^d,
\end{equation}
where $k,d\geq 1$ are integers, 
$\T^d~:=~ \R^d/\Z^d$, 
and 
$u:[0,\infty)\times \T^d\to\C$ is the unknown stochastic process. 
The term $F(u,\phi\xi)$ is a stochastic forcing and in this paper 
we treat the following cases: the \emph{additive noise}, i.e. 
\begin{equation}
\label{add-noise}
F(u,\phi\xi) = \phi\xi
\end{equation}
and the \emph{(linear) multiplicative noise}, i.e. 
\begin{equation}
\label{mult-noise}
F(u,\phi\xi) = u\cdot \phi\xi ,
\end{equation}
where the right-hand side of \eqref{mult-noise} is understood as an It\^o product\footnote{ 
The multiplicative noise given by the Stratonovich product $u \circ \phi\xi$ with  real-valued $\xi$  
is relevant in physical applications, as it  conserves the mass of $u$ 
(i.e. $t\mapsto \|u(t)\|^2_{L_x^2(\T^d)}$ is constant) almost surely. 
Our analysis can handle either the It\^{o} or the Stratonovich product, 
and we choose to work with the former for the sake of simpler exposition. 
}. 
Here, $\xi$ is a space-time white noise, i.e. a Gaussian stochastic process with correlation function 
$\EE[\xi(t,x)\xi(s,y)] = \delta(t-s) \delta(x-y)$, where $\delta$ denotes the Dirac delta function. 
We recall that the white noise is very rough: the spatial regularity of $\xi$ is less than $-\frac{d}{2}$. 
Since the linear Schr\"{o}dinger equation does not provide any smoothing properties, 
we consider instead a spatially smoothed out version $\phi\xi$, 
where $\phi$ is a linear operator from $L^2(\T^d)$ into $H^s(\T^d)$,  
on which we make certain assumptions, depending on 
whether we are working with  \eqref{add-noise} or \eqref{mult-noise}.

Our main goal in this paper is to prove local well-posedness  of SNLS 
with either additive or multiplicative noise  in the Sobolev space $H^s(\T^d)$, 
for any subcritical non-negative regularity $s$ (see below for the meaning of ``subcritical''). 
In this work, 
solutions to \eqref{SNLS} 
are understood as solutions to the 
\emph{mild formulation} 
\begin{equation}
\label{SNLS-mild}
u(t) = S(t) u_0 \pm i\int_0^t S(t-t') (|u|^{2k}u)(t')\,dt' - i \Psi(t)\ ,\ t\geq0\,,
\end{equation}
where $S(t) :=e^{-it\Delta}$ is the linear Schr\"{o}dinger propagator. The term $\Psi(t)$ is a  \emph{stochastic convolution}  
corresponding to the stochastic forcing $F(u,\phi\xi)$, 
see \eqref{sc:Psia} and \eqref{sc:Psim} below. 
Our local-in-time argument uses the Fourier restriction norm method introduced by Bourgain \cite{bourgain-93-1} 
and the periodic Strichartz estimates proved by Bourgain and Demeter~\cite{bourgain-demeter}. 
In establishing local well-posedness for the multiplicative SNLS, 
we also have to combine these tools with the truncation method 
used by de Bouard and Debussche 
\cite{debouard-debussche-nls, debouard-debussche-nls-mult, debouard-debussche-kdv}.
Moreover,  
by proving probabilistic a priori bounds on the mass and  energy of  solutions, 
we establish global well-posedness in 
(i)  $L^2(\T)$ for cubic nonlinearities (i.e. $k=1$) when $d=1$, and 
(ii) $H^1(\T^d)$ for all  defocusing
energy-subcritical nonlinearities -- see Theorem~\ref{gwp:add} and the preceding discussion for more details. 

Previously, de Bouard and Debussche \cite{ debouard-debussche-nls-mult, debouard-debussche-nls} 
 studied SNLS on $\R^d$. 
They considered noise $\phi\xi$ that is white in time but correlated in space, 
where $\phi$ is a smoothing operator from $L^2(\R^d)$ to $H^s(\R^d)$.  
They proved  global existence and uniqueness of mild solutions  in (i) $L^2(\R)$ 
for the one-dimensional cubic SNLS 
and (ii)  $H^1(\R^d)$ for  defocusing energy-subcritical SNLS. 
Other works related to SNLS on $\R^d$ include the works by 
Barbu, R\"{o}ckner, and Zhang \cite{barbu-rockner-zhang14, barbu-rockner-zhang16} 
and by Hornung \cite{hornung}. 

On the $\R^d$ setting, the arguments given in  \cite{ debouard-debussche-nls-mult, debouard-debussche-nls}  
use fixed point arguments in the space $C_tH_x^1 \cap L_t^pW_x^{1,q}([0,T]\times\R^d)$, 
for some $T>0$ and some suitable  $p,q\ge 1$.
\footnote{Here, $W^{s, r}(\T^d)$
denotes the $L^r$-based Sobolev space defined by the Bessel potential norm:
\[\| u\|_{W^{s, r}(\T^d)} := \| \jb{\nb}^s u \|_{L^r(\T^d)} = \big\| \F^{-1}( \jb{n}^s \ft u(n))\big\|_{\ell_n^r(\Z^d)},\]
where $\jb{n}:=\sqrt{1+|n|^2}$. 
When $r = 2$, we have $H^s(\T^d) = W^{s, 2}(\T^d)$.} 
In particular, they use the (deterministic) Strichartz estimates: 
\begin{equation}
\|S(t)f\|_{L_t^pL_x^q(\R\times\R^d)} \le C_{p,q} \|f\|_{L^2_x(\R^d)}, 
\end{equation}
where the pair $(p,q)$ is admissible, i.e. $\frac2p+ \frac{d}{q} = \frac{d}{2}$, $2\le p,q,\le \infty$, 
and $(p,q,d)\neq (2,\infty,2)$. 
On $\T^d$, 
Bourgain and Demeter \cite{bourgain-demeter} proved the $\ell^2$-decoupling conjecture,
 and as a corollary, 
the following periodic Strichartz estimates:  
\begin{equation}
\label{perStrich}
\big\|S(t)P_{\le N} f\big\|_{L_{t,x}^p([0,T]\times\T^d)} \le 
 C_{p,T,\varepsilon}  N^{\frac{d}{2}- \frac{d+2}{p} +\varepsilon} 
\|f\|_{L^2_x(\T^d)}\,.
\end{equation}
Here,  $P_{\le N}$ is the Littlewood-Paley projection onto frequencies $\{n \in \Z^d :  |n|\le N\}$, \linebreak
$p\ge \frac{2(d+2)}{d}$, and $\eps>0$ is an arbitrarily small quantity \footnote{More recently, Killip and Vi\c{s}an \cite{killip-visan} removed the arbitrarily small loss of $\varepsilon$ derivatives in 
 \eqref{perStrich} when $p>\frac{2(d+2)}{d}$. 
 However, we do not need this scale-invariant improvement in our results.}.
However, such Strichartz estimates are not strong enough for a fixed point argument in mixed Lebesgue spaces for the deterministic NLS on $\T^d$. 
To overcome this problem, we shall employ the Fourier restriction norm method 
by means of $X^{s,b}$-spaces defined via the norms
\begin{equation}
\norm{u}_{X^{s,b}}
:=\big\| \jb{n}^s\jb{\tau-|n|^2}^b\F_{t,x}(u)(\tau,n)\big\|_{L^2_\tau\l^2_n(\R\times \Z^d)}\,. 
\end{equation}
The indices $s,b\in\R$ measure the spatial and temporal regularities of functions $u\in X^{s,b}$, and 
$\mathcal{F}_{t,x}$ denotes Fourier transform of functions defined on $\R\times \T^d$.  
This harmonic analytic method was introduced by Bourgain \cite{bourgain-93-1} 
for the deterministic nonlinear Schr\"{o}dinger equation (NLS): 
\begin{align}
i\partial_t u - \Delta u  \pm |u|^{2k}u =0 \,.
\end{align}

\subsection{Main results}
We now state more precisely the problems considered here. Let 
$(\Omega, \mathcal{A},\{\mathcal{A}_t\}_{t\ge 0}, \mathbb{P})$ be a filtrated probability space. Let $W$ be the $L^2(\T^d)$-cylindrical Wiener process given by
\begin{equation}
W(t,x,\omega) := \sum_{n\in\Z^d} \beta_n(t,\omega) e_n(x),\,
\end{equation}
where $\{\beta_n\}_{n\in\Z^d}$ is a family of independent complex-valued Brownian motions 
associated with the filtration $\{\mathcal{A}_t\}_{t\ge 0}$ 
and 
$e_n(x):= \exp( 2\pi i n\cdot x)$,  $n\in\Z^d$.  
The space-time white noise $\xi$ 
is given by  the (distributional) time derivative of $W$, i.e.  $\xi=\frac{\partial W}{\partial t}$. 
Since the spatial regularity of $W$ is too low (more precisely, 
for each fixed $t\ge 0$, 
$W(t)\in H^{-\frac{d}{2}-\varepsilon}(\T^d)$
almost surely for any $ \varepsilon>0$), 
we consider a smoothed out version $\phi W$ as follows. 
Recall that a bounded linear operator  $\phi$ from a separable Hilbert space $H$ to a Hilbert space $K$ is  Hilbert-Schmidt  if 
\begin{equation}
\|\phi\|_{\mathcal{L}^2(H;K)}^2 := \sum_{n\in \Z^d} \| \phi h_n\|_{K}^2 <\infty\,,
\end{equation}
where $\{h_n\}_{n\in\Z^d}$ is an orthonormal basis of $H$ (recall that $\|\cdot\|_{\mathcal{L}^2(H;K)}$ does not depend on the choice of $\{h_n\}_{n\in\Z^d}$). 
Throughout this work, 
we assume $\phi\in \mathcal{L}^2(L^2(\T^d); H^s(\T^d))$ for appropriate $s\ge 0$. 
In this case, $\phi W$ is a Wiener process with sample paths in $H^s(\T^d)$ and 
its time derivative  $\phi\xi$   corresponds to a noise 
which is white in time and correlated in space (with correlation function depending on $\phi$). 
We can now define the stochastic convolution $\Psi(t)$ from \eqref{SNLS-mild} for 
(i) the additive noise \eqref{add-noise}: 
\begin{equation}
\label{sc:Psia}
\Psia(t) :=\int_0^t S(t-t') \phi \,dW(t')
\end{equation}
and (ii) the multiplicative noise \eqref{mult-noise}: 
\begin{equation}
\label{sc:Psim}
\Psim(t):= \Psi[u](t) := \int_0^t S(t-t') u(t') \phi \,dW(t')\,. 
\end{equation}  

\noi
We are now ready to state our first result. 
\begin{theorem}[Pathwise local well-posedness for additive SNLS]
	\label{lwp:add}
Given  $s>s_{\textup{crit}}$ non-negative, let $\phi\in \mathcal{L}^2(L^2(\T^d); H^s(\T^d))$ and 
$F(u,\phi)=\phi\xi$.  
Then for any $u_0\in H^s(\T^d)$, there exist a stopping time $T=T(\|u_0\|_{H^s},\Psi)$ that is almost surely positive, and	a unique adapted process $u \in   C([0,T];H^s(\T^d)) \cap X^{s,\frac12-\varepsilon}([0,T])$ solving SNLS with additive noise on $[0,T]$ almost surely, for some $\varepsilon>0$.	
\end{theorem}

Here, $X^{s,b}([0,T])$ is a time restricted version of  the $X^{s,b}$-space, see \eqref{Xsblocal} below.   
The proof of this result relies on a fixed point argument for \eqref{SNLS-mild} in a closed subset 
of   $X^{s,b}([0,T])$. 
We are required to use 
$b=\frac12-\varepsilon$ in order to capture the temporal regularity of $\Psi$.  
Since $X^{s,b}([0,T])$ does not embed into $C([0,T];H^s)$ when $b<\frac12$,  
 we need to prove the continuity in time of solutions a posteriori. 
Our local well-posedness result above (as well as Theorem~\ref{lwp:mult} below) covers all non-negative subcritical regularities. 

\begin{remark}\label{remark-critical}
We point out that  $s_{\textup{crit}}$ is negative only for the one-dimensional cubic NLS, 
i.e.  $(d,k)=(1,1)$ for which $s_{\textup{crit}}=-\frac12$.  
Below $L^2(\T)$, the deterministic cubic NLS on $\T$ was shown to be ill-posed.  
Indeed, Christ, Colliander and Tao \cite{CCT03instab}  and Molinet \cite{MolinetMRL} 
showed that the solution map $u_0\in H^s(\T) \mapsto u(t)\in H^s(\T)$ is discontinuous whenever $s<0$. 
More recently, Guo and Oh \cite{guo-oh} showed an even stronger ill-posedness result, 
in the sense that  for any $u_0\in H^s(\T)$, $s\in (-\frac18,0)$,  
 there is no distributional solution $u$ that is also a limit of smooth solutions in $C([-T,T]; H^s(\T))$.  
In the (super)critical regime, i.e. for $s\le -\frac12 =s_{\textup{crit}}$, 
Oh \cite{Oh17rmk} and Oh and Wang \cite{OhWangillposed} 
showed a norm inflation phenomenon at any $u_0\in H^s(\T)$: for any $\varepsilon>0$ and $u_0\in H^s(\T)$, 
there exists a solution $u^{\varepsilon}$ to NLS  such that 
$\|u^{\varepsilon}(0)-u_0\|_{H^s(\T)}<\eps$ and 
$\|u^{\varepsilon}(t)\|_{H^s(\T)}>\eps^{-1}$ 
for some  $t\in (0, \varepsilon)$. 
\end{remark}

\begin{remark}
Although we present our results for SNLS on the standard torus $\T^d=\R^d/\Z^d$, our arguments hold on any torus $\T_{{\boldsymbol\alpha}}^d=\prod_{j=1}^d\R/{\alpha_j}\Z\,$, where $\boldsymbol{\alpha}=(\alpha_1, ..., \alpha_d)\in [0,\infty)^d$. This is because the periodic Strichartz estimates \eqref{perStrich} of Bourgain and Demeter \cite{bourgain-demeter} hold for irrational tori ($\T_{{\boldsymbol\alpha}}^d$ is irrational  if there is no $\boldsymbol\gamma\in\mathbb Q^d$ such that $\boldsymbol{\gamma}\cdot\boldsymbol{\alpha}=0$). 
Prior to \cite{bourgain-demeter},
Strichartz estimates were harder to establish on irrational tori 
-- see \cite{guo-oh-wang} and references therein. 
\end{remark}

\begin{remark}
The deterministic NLS is   
locally well-posed in the critical space $H^{s_{\textup{crit}}}(\T^d)$, 
for almost all pairs $(d,k)$, except for the cases $(1,2), (2,1), (3,1)$ which are still open -- see \cite{BourgainIJM13,HTTduke11,HTTjram14,wang-nls}.  
In these papers, the authors employ the critical spaces $X^s, Y^s$ 
based on the spaces $U^2$, $V^2$ of Koch and Tataru \cite{KochTataruCPAM05}. 
We point out that Brownian motions belong almost surely to $V^p$, for $p>2$, 
but not  $V^2$ (hence neither to  $U^2$). 
Consequently,  the spaces  $X^s, Y^s$ are not suitable for obtaining local well-posedness of  SNLS.  
\end{remark}

Now let us recall the following conservation laws for the 
deterministic NLS: 
\begin{align}
\label{defn:mass}
M(u(t)) &:= \frac12\int_{\T^d} |u(t,x)|^2\,dx\\
\label{defn:energy}
E(u(t)) &:= \frac12 \int_{\T^d} |\nabla_x u(t,x)|^2 \pm \frac{1}{2k+2} \int_{\T^d} |u(t,x)|^{2k+2} \,dx, 
\end{align}
where the sign $\pm$ in \eqref{defn:energy} matches that in \eqref{SNLS} and \eqref{SNLS-mild}. 
Recall that SNLS \eqref{SNLS} with the  $+$ sign is called defocusing (and focusing for the $-$ sign). 
We say that SNLS is energy-subcritical if $s_{\textup{crit}}<1$ 
(i.e. for any $k\ge 1$ when $d=1,2$ and for $k=1$ when $d=3$). 
 
For solutions of SNLS these quantities are no longer necessarily conserved. 
However, It\^{o}'s lemma allows us to bound these in a probabilistic manner similarly to de Bouard and Debussche \cite{debouard-debussche-nls, debouard-debussche-nls-mult}.
Therefore, we obtain the following:

\begin{theorem}[Pathwise global well-posedness for additive SNLS]\label{gwp:add}
	Let $s\ge 0$.  Given $\phi\in\mathcal{L}^2(L^2(\T^d); H^s(\T^d))$, let $F(u,\phi)=\phi\xi$ 
	and $u_0\in H^s(\T^d)$. Then the $H^s$-valued solutions of Theorem~\ref{lwp:add} extend globally in time almost surely in the following cases:
	
	\vspace{.1cm}
	
	\noi
	\textup{ (i)} the (focusing or defocusing) one-dimensional cubic SNLS for all $s\ge 0$; 
	
	\noi
	\textup{\ (ii)} the defocusing energy-subcritical SNLS for all $s\ge 1$. 
	
\end{theorem}

We now move onto the problem with multiplicative noise, i.e. SNLS with \eqref{mult-noise}. For this case, we need a stronger assumption on $\phi$. By a slight abuse of notation, for a bounded linear operator $\phi$ from $L^2(\T^d)$ to a Banach space $B$, we say that $\phi\in\L^2(L^2(\T^d); B)$ if\footnote{In fact, 
such operators are known as nuclear operators of order 2 
and their introduction goes back to the work of A.~Grothendieck 
on nuclear locally convex spaces.} 
\begin{align*}
	\|\phi\|_{\mathcal{L}^2(L^2(\T^d);B)}^2 := \sum_{n\in \Z^d} \| \phi e_n\|_{B}^2 <\infty\,.
\end{align*}
For $s\in\R$ and $r\ge 1$, 
we also define the Fourier-Lebesgue space $\F L^{s,r}(\T^d)$ via the norm
\begin{align*}
	\norm{f}_{\F L^{s,r}(\T^d)}:=\big\|\jb{n}^s\ft{f}(n)\big\|_{\l^r_n(\Z^d)}\,.
\end{align*}
Clearly, when $r=2$ we have $\F L^{s,r}(\T^d) = H^s(\T^d)$ and 
for $s_1\le s_2$ and $r_1\le r_2$ we have 
$\F L^{s_2,r_1}(\T^d)\subset \F L^{s_1,r_2}(\T^d)$.  
We now state our local well-posedness result for the multiplicative SNLS.
\begin{theorem}[Local well-posedness for multiplicative SNLS]
	\label{lwp:mult}
Given  $s>s_{\textup{crit}}$ non-negative, let $\phi\in \mathcal{L}^2(L^2(\T^d); H^s(\T^d))$. 
 If $s\le \frac{d}{2}$, we further impose that 
\begin{align}
\label{phiFLsr}
\phi\in \mathcal{L}^2(L^2(\T^d); \F L^{s,r}(\T^d))
\end{align}
for some $r\in\big[1,\frac{d}{d-s}\big)$ when $s>0$ and $r=1$ when $s=0$.    
Let $F(u,\phi)=u\cdot \phi\xi$. 
	Then 
	for any $u_0\in H^s(\T^d)$, there exist 
	a stopping time $T$ that is almost surely positive, 
	and 
	a unique adapted process 
\begin{align}\label{L2Omega-space}
u \in L^2\big(\Omega; C([0,T];H^s(\T^d))  \cap X^{s,b}([0,T]) \big)
\end{align} 
	solving SNLS 
	with multiplicative noise. 
\end{theorem}

\begin{remark}
If $\phi\xi$  is a spatially homogeneous noise, i.e. $\phi$ is translation invariant, 
then the extra assumption \eqref{phiFLsr} is superfluous. 
Indeed, if $\widehat{ \phi e_n}(m) =0$, for all $m,n\in\Z^d$, $m\neq n$ and  $\phi\in\mathcal{L}^2(L^2(\T^d); H^s(\T^d))$, 
then $\phi\in \L^2(L^2(\T^d); \F L^{s,r}(\T^d))$ for any $r\ge 1$. 

We point out that an extra condition in the multiplicative case was also used by 
de~Bouard and Debussche \cite{debouard-debussche-nls} 
in their study of SNLS in $H^1(\R^d)$, 
namely they required that $\phi$ is 
a $\gamma$-radonifying operator from $L^2(\R^d)$ into $W^{1,\alpha}(\R^d)$ 
 for some appropriate $\alpha$, 
 as compared to the requirement that 
 $\phi$ is Hilbert-Schmidt from $L^2(\R^d)$ into $H^s(\R^d)$ in the additive case.   
\end{remark}

In the multiplicative case, the stochastic convolution depends on the solution 
$u$ and this forces us to work in the space in \eqref{L2Omega-space}. 
In order to control the nonlinearity in this space, 
we use a truncation method which has been used for SNLS on $\R^d$ 
by de Bouard and Debussche \cite{debouard-debussche-nls, debouard-debussche-nls-mult}.  
Moreover, we combine this method with the use of $X^{s,b}$-spaces in a similar manner as in
\cite{debouard-debussche-kdv}, where the same authors studied the stochastic KdV equation with low regularity initial data on $\R$. 
This introduces some technical difficulties which did not appear 
when using the more classical Strichartz spaces as those used in \cite{debouard-debussche-nls, debouard-debussche-nls-mult}.

Next, we prove global well-posedness of SNLS \eqref{SNLS} with multiplicative noise. 
Similarly to the additive case, the main ingredient is the probabilistic a priori bound on the mass and energy of a local solution $u$. However, 
we further need to obtain uniform control on the $X^{s,b}$-norms for solutions to  truncated versions of 
\eqref{SNLS-mild}. 

\begin{theorem}[Global well-posedness for multiplicative SNLS]\label{gwp:mult}
	Let $s\ge 0$.  Given $\phi$ with the same assumptions as in Theorem \ref{lwp:mult}, let $F(u,\phi)=u\cdot\phi\xi$ 
	and $u_0\in H^s(\T^d)$. Then the $H^s$-valued solutions of Theorem~\ref{lwp:mult} extend globally in time in the following cases:
	
	\vspace{.1cm}
	
	\noi
	\textup{ (i)} the (focusing or defocusing) one-dimensional cubic SNLS for all $s\ge 0$; 
	
	\noi
	\textup{\ (ii)} the defocusing energy-subcritical SNLS for all $s\ge 1$. 
\end{theorem}

Before concluding this introduction let us state two remarks. 

\begin{remark} 
We point out that Theorem~\ref{lwp:add} and Theorem~\ref{lwp:mult} are almost optimal 
for handling the regularity of initial data 
since the deterministic NLS is ill-posed for $s<s_{\textup{crit}}$ (see Remark \ref{remark-critical}). 
In terms of the regularity of the noise, at least in the additive noise case, 
it is possible to consider rougher noise 
by employing the Da Prato-Debussche trick, 
namely by writing a solution $u$ to \eqref{SNLS-mild} as $u= v+\Psi$ 
and considering the equation for the residual part $v$. 
In general, this procedure allows one to treat rougher noise, 
see for example 
\cite{BOP1, BOP2, CollianderOh}. 
where they treat NLS with rough random initial data. 
In the periodic setting however, 
the argument gets more complicated 
(see for example \cite{BOP1, BOP2} on $\R^d$ versus \cite{CollianderOh,NahmodStaffilani} on $\T^d$). 
The actual implementation of the aforementioned trick requires 
cumbersome case-by-case analysis where the number of cases grows exponentially in $k$. 
Even for the cubic case on $\T^d$ the analysis is involved, whereas 
on $\R^d$ one can use bilinear Strichartz estimates which are not available on $\T^d$.
\end{remark}

\begin{remark}
In the multiplicative noise case, there are well-posedness results 
on a general compact Riemannian manifold $M$ without boundaries. 
In \cite{BrzezniakMillet}, 
Brze\'zniak and Milllet  use the Strichartz estimates of \cite{BGT} and the standard space-time Lebesgue spaces (i.e. 
without the Fourier restriction norm method). 
For $M=\T^d$, Theorem~\ref{lwp:mult} improves the result in \cite{BrzezniakMillet} 
since it requires less regularity on the noise and initial data.  
In \cite{brezniak-h-w}, 
Brze\'zniak, Hornung, and Weiss construct martingale solutions in $H^1(M)$ 
for the multiplicative SNLS 
with energy-subcritical defocusing nonlinearities 
and mass-subcritical focusing nonlinearities. 
\end{remark}

\subsection*{Organization of the paper}
In Section \ref{sect:Xsb}, we provide some preliminaries for the Fourier restriction norm method 
and prove the multilinear estimates necessary for the local well-posedness results. 
In Section~\ref{sect:stochests}, 
we prove some properties of the stochastic convolutions $\Psi$ and $\Psi[u]$ 
given respectively by \eqref{sc:Psia} and  \eqref{sc:Psim}. 
We prove  Theorems \ref{lwp:add} and \ref{lwp:mult} in Section~\ref{sect:LWP}. 
Finally, in Section \ref{sect:GWP} we prove the  global results Theorems \ref{gwp:add} and \ref{gwp:mult}. 

\subsection*{Notations} 
Given $A,B\in\R$, 
we use the notation $A\lesssim B$ to mean $A\le CB$ for some  constant $C\in (0,\infty)$ and write $A\sim B$ to mean $A\lesssim B$ and $B\lesssim A$. 
We sometimes emphasize any dependencies of the implicit constant as subscripts on $\lesssim$, $\gtrsim$, 
\linebreak 
and $\sim$; e.g. $A\lesssim_{p} B$ means $A\le C B$ for some constant $C=C(p)\in (0,\infty)$ that depends on the parameter $p$. We denote by $A\wedge B$ and $A\vee B$ the minimum and maximum between the two quantities respectively. Also, $\lceil A\rceil$ denotes the smallest integer greater or equal to $A$, while $\lfloor A\rfloor$ denotes the largest integer less than or equal to $A$. 

Given a function $g:U\to \C$, where $U$ is either $\T^d$ or $\R$, our convention of the Fourier transform of $g$ is given by
\[\ft{g}(\xi)=\int_U e^{2\pi i \xi\cdot x}g(x)\,dx\,,\]
where $\xi$ is either an element of $\Z^d$ (if $U=\T^d$) or an element of $\R$ (if $U=\R$). For the sake of convenience, we shall omit the $2\pi$ from our writing since it does not play any role in our arguments. 

For $c\in\R$, we sometimes write $c+$ to denote $c+\eps$ for sufficiently small $\eps>0$, and write $c-$ for the analogous meaning. For example, the statement `$u\in X^{s,\frac{1}{2}-}$' should be read as `$u\in X^{s,\frac12-\eps}$ for sufficiently small $\eps>0$'. 

For the sake of readability, in the proofs we sometimes omit the underlying domain $\T^d$ from various norms, e.g. 
we write $\|f\|_{H^s}$ instead of $\|f\|_{H^s(\T^d)}$ and 
$\|\phi\|_{\L^2(L^2;H^s)}$ instead of $\|\phi\|_{\L^2(L^2(\T^d);H^s(\T^d))}$.

\subsection*{Acknowledgements}
The authors would like to thank their advisors, Tadahiro Oh and Oana Pocovnicu,  
for suggesting this problem and their continuous support throughout this work, 
as well as Professor  Yoshio~Tsutsumi, 
Yuzhao~Wang and Dimitrios~Roxanas for several useful discussions on the present paper.  

The authors were supported by 
The Maxwell Institute Graduate School in Analysis and its Applications, 
a Centre for Doctoral Training funded by 
the UK Engineering and Physical Sciences Research Council (grant EP/L016508/01), 
the Scottish Funding Council, Heriot-Watt University and the University of Edinburgh.

\section{Fourier restriction norm method}
\label{sect:Xsb}
Let $s,b\in\R$. The Fourier restriction norm space $X^{s,b}$ adapted to the Schr\"odinger equation on $\T^d$ 
is the space of tempered distributions $u$ on $\R\times\T^d$ such that the norm
\begin{equation*}
\norm{u}_{X^{s,b}}
:=\norm{\jb{n}^s\jb{\tau-|n|^2}^b\F_{t,x}(u)(\tau,n)}_{\l^2_nL^2_\tau(\Z^d\times \R)}
\end{equation*}
is finite. Equivalently, the $X^{s,b}$-norm can be written in its interaction representation form:
\begin{equation}
	\norm{u}_{X^{s,b}}=\norm{\jb{n}^s\jb{\tau}^b\F_{t,x}\brac{S(-t)u(t)}(n,\tau)}_{\l^2_nL^2_\tau(\Z^d\times \R)}\,,\label{Xsb-interact-rep}
\end{equation}
where $S(t)=e^{-it\Dl}$ is the linear Schr\"odinger propagator. 
We now state some facts on $X^{s,b}$-spaces.  
The interested reader can find the proof of these and further properties in \cite{tao-book}.  
Firstly, we have the following continuous embeddings 
		\begin{align}
		X^{s,b} &\hookrightarrow C(\R; H_x^s(\T^d)) \ \mbox{, for } b>\frac{1}{2}\,,\\
		X^{s',b'}&\hookrightarrow X^{s,b}\  \mbox{, for } s'\ge s \mbox{ and } b'\ge b\,.
		\end{align}
We have the duality relation 
\begin{equation}
\label{Xsbduality}
\norm{u}_{X^{s,b}}=\sup_{\norm{v}_{X^{-s,-b}}\le 1}\left|\int_{\R\times\T^d} u(t,x) \overline{v(t,x)}\,dt\,dx\right|\,.
\end{equation}
\begin{lemma}[Transference principle, {\cite[Lemma~2.9]{tao-book}}]\label{trans-prin}
	\label{TransfPr}
	Let $Y$ be a Banach space of functions on $\R\times \T^d$ such that 
	\begin{equation*}
	\|e^{it\lambda} e^{\pm it\Delta} f\|_Y \les \|f\|_{H^s(\T^d)}
	\end{equation*}
	for all $\lambda\in\R$ and all $f\in H^s(\T^d)$. Then, for any $b>\frac12$, 
	\begin{equation*}
	\|u\|_Y \les \| u\|_{X^{s,b}}\, .
	\end{equation*}
\end{lemma}
Given a time interval $I\subseteq\R$, 
one defines the time restricted space $X^{s,b}(I)$ via the norm
\begin{equation}
\label{Xsblocal}
\xsbt{u}{s}{b}{I}:=\inf\left\{\|{\tilde u}\|_{X^{s,b}}: \tilde u|_{I} = u\right\}.
\end{equation}
We note that for $s\geq0$ and $0\le b<\frac{1}{2}$, 
we have 
\begin{align}
\label{Xsblocalsim}
\norm{u}_{X^{s,b}(I)}\sim\xsb{\one_{I}(t)u(t)}{s}{b}\,,
\end{align}
see for example \cite[Lemma 2.1]{debouard-debussche-kdv} for a proof (for $X^{s,b}$ spaces adapted to the KdV equation). 
\begin{lemma}[Linear estimates, {\cite[Proposition~2.12]{tao-book}}]
\label{lem:linests}
Let $s\in\R$ and suppose $\eta$ is smooth and compactly supported.   
Then, we have
\begin{align}
\|\eta(t) S(t)f\|_{X^{s,b}} &\les \|f\|_{H^s(\T^d)} \ \text{, for }b\in\R\,;\\
\left\| \eta(t) \int_0^t S(t-t') F(t')dt' \right\|_{X^{s,b}} &\les \|F\|_{X^{s,b-1}} \ \text{, for  }b>\frac12\,.
\end{align}
\end{lemma}

\noi
By localizing in time, we can gain a smallness factor, as per lemma below. 
		\begin{lemma}[Time localization property, {\cite[Lemma~2.11]{tao-book}}]\label{xsb-time-loc}
			Let $s\in\R$ and $-\frac{1}{2}<b'<b<\frac{1}{2}$. For any $T\in (0,1)$, we have
			\[\xsbt{f}{s}{b'}{[0,T]}\les_{b,b'} T^{b-b'}\xsbt{f}{s}{b}{[0,T]}\,.\]
		\end{lemma}

We now give the proofs of the multilinear estimates necessary to control the nonlinearity $|u|^{2k}u$. Recall the $L^4$-Strichartz estimate due to Bourgain \cite{bourgain-93-1} 
(see also \cite[Proposition~2.13]{tao-book}): 
\begin{equation}
\label{L4Strichartz}
\|u\|_{L^4_{t,x}(\R\times\T)} \lesssim \|u\|_{X^{0,\frac38}} .
\end{equation}

\begin{lemma}\label{trilin}
Let $d=1$,  $s\geq0$, $b\geq \frac38$, and  $b'\leq\frac58$. Then, for any time interval $I\subset\R$, 
we have 
\begin{equation}
\label{trilinStrich}
\left\| u_1 \overline{u_2} u_{3} \right\|_{X^{s,b'-1}(I)} 
\lesssim 
 \prod_{j=1}^{3} \|u_j\|_{X^{s,b}(I)} .
\end{equation}
\end{lemma}
\begin{proof}
By the triangle inequality it suffices to prove \eqref{trilinStrich} for $s=0$. 
We claim that
\begin{equation*}
\left| \int_{\R\times\T^d} u_1 \overline{u_2} u_3 \overline{v}\,dxdt \right| 
\lesssim \prod_{j=1}^{3} \|u_j\|_{X^{0,b}} \|v\|_{X^{0,1-b'}}
\end{equation*}
for any factors $u_1, u_2, u_3, v$. Indeed, this follows immediately from H\"{o}lder inequality 
and \eqref{L4Strichartz} for each of the four factors (hence the restrictions $b,1-b'~\geq~\frac38$). 
Thus,  the global-in-time version of \eqref{trilinStrich}, i.e. $I=\R$, 
follows by the duality relation \eqref{Xsbduality}. 
For an arbitrary time interval $I$, if $\tilde{u}_j$ is an extension of $u_j$, $j=1,2,3$, then 
$\tilde{u}_1 \overline{\tilde{u}_2} \tilde{u}_3$ is an extension of $u_1 \overline{u_2} u_3$. 
We use the previous step to get
$$ \left\| u_1 \overline{u_2} u_{3} \right\|_{X^{s,b'-1}(I)} \le 
\left\| \tilde{u}_1 \overline{\tilde{u}_2} \tilde{u}_{3} \right\|_{X^{s,b'-1}} \les 
\prod_{j=1}^{3} \|\tilde{u}_j\|_{X^{0,b}} $$
and then we take infimum over all extensions $\tilde{u}_j$'s 
and \eqref{trilinStrich} follows. 
\end{proof}

Due to the scaling and  Galilean symmetries of the linear Schr\"{o}dinger equation, 
the periodic  Strichartz estimate \eqref{perStrich} of 
Bourgain and Demeter \cite{bourgain-demeter} 
is equivalent with 
\begin{equation}
\label{StrichartzKV}
\|S(t) P_Q  f\|_{L^p_{t,x}(I\times\T^d)} \lesssim_{|I|} |Q|^{\frac{1}{2} - \frac{d+2}{pd}+} \|f\|_{L^2_x(\T^d)}, 
\end{equation}
for any $d\geq 1$, $p\ge\frac{2(d+2)}{d}$, $I\subset\R$ finite time interval, and $Q\subset\R^d$ dyadic cube. 
Here, $P_Q$ denotes the frequency projection onto $Q$, i.e. $\widehat{P_Qf}(n) = \ind_Q(n) \widehat{f}(n)$. 
By the transference principle (Lemma~\ref{TransfPr}), we get 
\begin{equation}
\label{StrichartzKVtp}
\|P_{Q}u\|_{L^{p}_{t,x}(I \times\T^d)} \lesssim_{|I|} |Q|^{\frac{1}{2} - \frac{d+2}{pd} +} \| u\|_{X^{0, b}(I)} ,
\end{equation}
for any $b>\frac12$. 
By interpolating \eqref{StrichartzKVtp} 
with
\begin{equation}
\|P_{Q}u\|_{L^{p}_{t,x}(I \times\T^d)} \lesssim |Q|^{\frac12-\frac1p} \| u\|_{X^{0, \frac12-\frac1p}(I)} ,
\end{equation}
(which follows immediately from Sobolev inequalities, \eqref{Xsb-interact-rep},  
and the $H^s(\T^d)$-isometry of $S(-t)$), 
we can lower the time regularity from $b=\frac12+\delta$ to $\tilde{b}=\frac12-\delta$, 
for sufficiently small~$\delta>0$. 
Thus, we also have
\begin{equation}
\label{StrichartzKVtpinterp}
\|P_{Q}u\|_{L^{p}_{t,x}(I\times\T^d)} \lesssim_{|I|,\delta}  
 |Q|^{\frac{1}{2} - \frac{d+2}{pd} + o(\delta)} 
  \|u\|_{X^{0,\frac12-\delta}(I)} 
\end{equation}

Lemma \ref{trilin} only treats the cubic nonlinearity when $d=1$. 
We now prove the following 
general multilinear estimates to treat other cases. 
The proof borrows techniques from \cite{guo-oh-wang}.

\begin{lemma}\label{multilin}
Let $d,k\geq 1$ such that $dk\geq 2$ and let $I\subset\R$ be a finite time interval.  
Then for any $s>s_{\textup{c}}$,  
there exist $b=\frac12-$  and $b'=\frac12+$ such that 
\begin{equation}
\label{multilinStrich}
\left\| u_1 \overline{u_2} \cdots \overline{u_{2k}} u_{2k+1} \right\|_{X^{s,b'-1}(I)} 
\lesssim_{|I|}
 \prod_{j=1}^{2k+1} \|u_j\|_{X^{s,b}(I)} .
\end{equation}
\end{lemma}
\begin{proof}
In view of \eqref{Xsblocalsim}, 
we can assume that $u_j(t)=\one_I(t) u_j(t)$
and thus by the duality relation \eqref{Xsbduality},  
it suffices to show 
\begin{equation}
\label{multilin-dual}
\left|
\int_{\R\times \T^{d}} \big(\jb{\nabla}^s (u_1 \overline{u_2} \cdots u_{2k+1})\big) \overline{v} \,dxdt \right|
   \lesssim \|v\|_{X^{0,1-b'}}  \prod _{j=1}^{2k+1}\|u_j\|_{X^{s,b}} . 
\end{equation}
We use Littlewood-Paley decomposition: 
we estimate the left-hand side of \eqref{multilin-dual} 
when $v=P_Nv$, $u_j=P_{N_j}u_j$ 
for some dyadic numbers $N, N_j\in 2^{\Z}$, $1\le j\le 2k+1$.  
Then the claim follows by triangle inequality and performing the summation
\begin{equation}
\label{summation}
\sum_{N_1} \sum_{\substack{N\\N\lesssim N_1}}\sum_{\substack{N_2\\N_2\leq N_1}} \cdots 
   \sum_{\substack{N_{2k+1}\\N_{2k+1}\leq N_{2k}}} .
\end{equation}
Notice that 
without loss of generality, we may assume that 
$N_1\geq N_2\geq\ldots\geq N_{2k+1}$, 
in which case we also have $N\lesssim N_1$, and that 
 the factors $v$ and $u_j$ are real-valued and non-negative. 

Let $\varepsilon :=s-s_{\textup{c}}$, and we distinguish two cases. 

\noindent 
\textbf{Case 1:} $N_1\sim N_2$. 
By H\"{o}lder inequality, 
\begin{equation}
\label{multilin-dual-LP}
N^s \int_{\R\times \T^{d}}  u_1 {u_2} \cdots u_{2k+1} {v} \,dxdt
\lesssim
N_1^{\frac{s}{2}} \|u_1\|_{L^q_{t,x}} N_2^{\frac{s}{2}} \|u_2\|_{L^q_{t,x}} \prod_{j=3}^{2k+1} \|u_j\|_{L^p_{t,x}}  \|v\|_{L^r_{t,x}} , 
\end{equation}
with $p,q,r$ chosen such that $\frac{2k-1}{p}+\frac{2}{q} +\frac1r=1$. 
We take $p, q$ 
such that $\frac{d}{2}-\frac{d+2}{p}=s_{\textup{crit}}$ and $\frac{d}{2}-\frac{d+2}{q}=\frac12 s_{\textup{crit}}$, 
or equivalently $p=k(d+2)$ and $q=\frac{4k(d+2)}{dk+2}$. 
These give the H\"{o}lder  exponent $r= \frac{2(d+2)}{d}$. 
By \eqref{StrichartzKVtpinterp} and \eqref{StrichartzKVtp}, we get 
\begin{align}
\label{multilin-eqn1-case1}
 N_j^{\frac{s}{2}} \|u_j\|_{L^q_{t,x}} &\lesssim N_j^{-\frac{\varepsilon}{2}+} \|u_j\|_{X^{s,b}} , \quad j=1,2\\
\label{multilin-eqn2-case1}
\|u_j\|_{L^p_{t,x}} &\lesssim N_j^{-\varepsilon+} \|u_j\|_{X^{s,b}},\quad 3\leq j\leq 2k+1, \\
\label{multilin-eqn3-case1}
\|v\|_{L^r_{t,x}} &\lesssim N^{0+} \|v\|_{X^{0,1-b'}} .
\end{align}
By choosing $\delta, \delta' \ll \varepsilon$ in $b:=\frac12-\delta$ and in $1-b'=\frac12-\delta'$, respectively, 
we get 
\begin{equation}
\textup{RHS of }\eqref{multilin-dual-LP} 
 \lesssim N^{-\frac{\varepsilon}{4}}\|v\|_{X^{0,1-b'}} \prod_{j=1}^{2k+1} N_j^{-\frac{\varepsilon}{4}} \|u_j\|_{X^{s,b}}  .
\end{equation}
The factors $N^{-\frac{\varepsilon}{4}}$, $N_j^{-\frac{\varepsilon}{4}}$ guarantee that we can perform \eqref{summation}.

\noindent 
\textbf{Case 2:} $N_1\gg N_2$. Then, we necessarily have $N_1\sim N$ or else the left hand side of \eqref{multilin-dual} vanishes. 
By H\"{o}lder inequality, 
\begin{equation}
N^s \int_{\R\times \T^{d}}  u_1 {u_2} \cdots u_{2k+1} {v} \,dxdt 
\les 
N_1^s \|u_1\|_{L^q_{t,x}} \prod_{j=2}^{2k+1} \|u_j\|_{L^p_{t,x}}  \|v\|_{L^r_{t,x}} , 
\end{equation}
with $\frac{2k}{p}+\frac{1}{q}+\frac1r=1$. 
As in Case 1, we would like to have $p$ such that 
$\frac{d}{2}-\frac{d+2}{p}= s_{\textup{crit}}$,
or equivalently $p=k(d+2)$. 
However, the best we can do with the Strichartz estimate for the remaining factors 
is to choose $q=r=\frac{2(d+2)}{d}$, 
so that we have
\begin{align}
\label{multilin-eqn1}
 N_1^s \|u_1\|_{L^q_{t,x}} &\lesssim N_1^{0+} \|u_1\|_{X^{s,b}} ,\\
\label{multilin-eqn2}
\|u_j\|_{L^p_{t,x}} &\lesssim N_j^{-\varepsilon+} \|u_j\|_{X^{s,b}},\quad 2\leq j\leq 2k+1, \\
\label{multilin-eqn3}
\|v\|_{L^r_{t,x}} &\lesssim N_1^{0+} \|v\|_{X^{0,1-b'}} .
\end{align}
Notice that we can overcome the loss of derivative $N_1^s$ 
only up to a logarithmic factor. 
We need a slightly refined analysis. 

We cover the dyadic frequency annuli of $u_1$ and of $v$ with dyadic cubes  of side-length $N_2$, 
i.e. 
$$\{\xi_1 : |\xi_1|\sim N_1\}\subset \bigcup_{\ell} Q_{\ell} \quad, \quad 
  \{\xi : |\xi|\sim N\}\subset \bigcup_{j} R_j \,.$$
There are approximately $\left(\frac{N_1}{N_2}\right)^d$-many cubes needed, and so 
$$u_1=\sum_{\ell} P_{Q_{\ell}} u_1 =: \sum_{\ell} u_{1,\ell}\quad, \quad 
  v= \sum_j P_{R_j}v =: \sum_j v_j$$
are decompositions into finitely many terms. 
Since 
$|\xi_1-\xi| \lesssim N_2$ for $\xi_1\in \supp(\widehat{u_1}), \xi\in \supp(\widehat{v})$ 
on the convolution hyperplane, 
there exists a constant $K$ such that 
if $\mathrm{dist}(Q_{\ell}, Q_j)>KN_2$, then  the integral in \eqref{multilin-dual} vanishes. 
Hence the summation \eqref{summation} is replaced by
\begin{equation}
\label{summation2}
\sum_{N_1} \sum_{\substack{N_2\\N_2\ll N_1}} \cdots 
   \sum_{\substack{N_{2k+1}\\N_{2k+1}\leq N_{2k}}} \sum_{\substack{\ell,j\\j\approx \ell}} .
\end{equation} 
Also, in place of \eqref{multilin-eqn1}-\eqref{multilin-eqn2}, we now have 
\begin{align}
\label{multilin-eqn11}
 N_1^s \|u_{1,\ell}\|_{L^q_{t,x}} &\lesssim N_2^{0+} \|u_{1,\ell}\|_{X^{s,b}} ,\\
 \label{multilin-eqn22}
\|u_i\|_{L^p_{t,x}} &\lesssim N_i^{-\varepsilon+} \|u_i\|_{X^{s,b}},\quad 2\leq i\leq 2k+1, \\
\label{multilin-eqn33}
\|v_j\|_{L^q_{t,x}} &\lesssim N_2^{0+} \|v_j\|_{X^{0,1-b'}} ,
\end{align}
Therefore, 
by Cauchy-Schwartz inequality and Plancherel identity, 
\begin{align*}
\textup{LHS of }\eqref{multilin-dual} &\lesssim 
   \sum_{N_2} \sum_{\substack{N_1\\N_1\gg N_2}} \sum_{\substack{\ell,j\\ \ell\approx j}} 
     N_2^{-\varepsilon+} \|u_{1,\ell}\|_{X^{s,b}} \|v_j\|_{X^{0,1-b'}}  \prod_{i=2}^{2k+1} \|u_i\|_{X^{s,b}}\\
    &\lesssim \sum_{N_2} N_2^{-\varepsilon+} 
      \left( \sum_{\substack{N_1\\N_1\gg N_2}} \sum_{\ell} \|u_{1,\ell}\|_{X^{s,b}}^2\right)^{\frac12} 
      \left( \sum_{\substack{N\\N\gg N_2}} 
       \sum_{j} \|v_j\|_{X^{0,1-b'}}^2 \right)^{\frac12}  \prod_{i=2}^{2k+1} \|u_i\|_{X^{s,b}}\\
     &\lesssim \sum_{N_2} N_2^{-\varepsilon+} \|u_1\|_{X^{s,b}} \|v\|_{X^{0,1-b'}} 
      \prod_{i=2}^{2k+1} \|u_i\|_{X^{s,b}}\\
     &\lesssim \prod_{i=1}^{2k+1} \|u_i\|_{X^{s,b}} \|v\|_{X^{0,1-b'}}
\end{align*}
and the proof is complete. 
\end{proof}

\section{The stochastic convolution}
\label{sect:stochests}
In this section, we prove some $X^{s,b}$-estimates on the stochastic convolution $\Psi(t)$ given 
either by \eqref{sc:Psia} or \eqref{sc:Psim}.
We first record the following Burkholder-Davis-Gundy inequality, which is a consequence of \cite[Theorem 1.1]{CR_Burkholder}.
\begin{lemma}[Burkholder-Davis-Gundy inequality]
\label{BDG}Let $H,K$ be separable Hilbert spaces, $T>0$, and $W$ is an $H$-valued Wiener process on $[0,T]$. 
Suppose that $\{\psi(t)\}_{t\in[0,T]}$ is an adapted process taking values in 
$\mathcal{L}^2(H;K)$. Then for $p\ge 1$,
\[\EE\left[\supT\norm{\int_0^t\psi(t')\,dW(t')}_K^p\right]\les_p 
  \EE\left[\left(\int_0^T \norm{\psi(t')}_{\mathcal{L}^2(H;K)}^2\,dt'\right)^\frac{p}{2}\right]\,.\]
\end{lemma}

In addition, we prove that $\Psi(t)$ is pathwise continuous in both cases. 
To this end,  we employ the factorization method of Da Prato \cite[Lemma 2.7]{daprato-kol-eqns-04}, 
i.e. we make use of the following lemma and \eqref{factorisation} below. 

\begin{lemma}
\label{fact-meth}
Let $H$ be a Hilbert space, $T>0$, $\alpha\in(0,1)$,  and $\sigma>\big(\frac{1}{\alpha},\infty\big)$. 
Suppose that ${f\in L^{\sigma}([0,T];H)}$. 
Then the function
\begin{equation}
F(t):=\intud{0}{t}{S(t-t')(t-t')^{\alpha-1}f(t')}{t'}\,,\quad t\in [0,T]
\end{equation}
	belongs to $C([0,T];H)$. Moreover,
\begin{equation}
\label{lem3p2:est}
\sup_{t\in[0,T]}\norm{F(t)}_{H}\lesssim_{\sigma,T}\norm{f}_{L^{\sigma}([0,T];H)}.
\end{equation}
\end{lemma}

We make use of the above lemma in conjunction with the following fact:
\begin{equation}
\label{factorisation}
\intud{\mu}{t}{(t-t')^{\alpha-1}(t'-\mu)^{-\alpha}}{t'}=\frac{\pi}{\sin(\pi\alpha)}\,,
\end{equation}
for all $0<\alpha <1$ and all  $0\le \mu <t$. 
This can be seen via considerations with Euler-Beta functions, see \cite{daprato-kol-eqns-04}.

We now treat the additive and multiplicative cases separately below in Subsection \ref{subsect:sto-conv-add} and \ref{subsect:sto-conv-mult} respectively. The arguments for the two cases are similar, albeit with some extra technicalities in the multiplicative case.

\subsection{The additive stochastic convolution}\label{subsect:sto-conv-add}
By Fourier expansion, the stochastic convolution \eqref{sc:Psia} for the additive noise problem can be written as
\begin{equation}
\label{Psia}
\Psia(t)  =  \sum_{n\in\Z^d} e_n
    \sum_{j\in\Z^d}  \widehat{(\phi e_j )}(n)  \int_0^t e^{i(t-t')|n|^2}d\beta_j(t')\, .
\end{equation}   
We first prove the following $X^{s,b}$-estimate on $\Psi$:
\begin{lemma}
	\label{stoc-conv-est-add}
Let $s\ge 0$, $0\le b<\frac{1}{2}$, $T>0$, and $\sigma\in [2,\infty)$. 
Assume that 
$\phi\in \mathcal{L}^2(L^2(\T^d); H^s(\T^d))$. 
Then for $\Psi$ given by \eqref{Psia} we have
\begin{align}
\EE\left[  \|\Psia\|^{\sigma}_{X^{s,b}{([0,T])}} \right]
	&\lesssim T^\frac{\sigma}{2}(1+T^2)^\frac{\sigma}{2}
	\|\phi\|_{\L^2(L^2(\T^d);H^s(\T^d))}^{\sigma}\,.
\end{align}
\end{lemma}
\begin{proof}
	Since $\one_{[0,T]}(t)\one_{[0,T]}(t')=\one_{[0,T]}(t)=1$ whenever $0\le t'\le t\le T$, we have 
	\begin{align*}
	\one_{[0,T]}(t)\Psia(t)(x)  &= 
	\sum_{n\in\Z^d} e_n \sum_{j\in\Z^d} 
	\ft{\phi e_j}(n) \one_{[0,T]}(t)  e^{it|n|^2}
	 \int_0^t \one_{[0,T]}(t') e^{-it'|n|^2}{d}\beta_j(t')
\end{align*}
By \eqref{Xsblocalsim}, we have
\begin{align}
	\xsbt{\Psia(t)}{s}{b}{[0,T]}&\sim\xsb{\one_{[0,T]}(t)\Psia(t)}{s}{b}\notag\\
	&= \|\jb{n}^s\jb{\tau}^b\F_{t,x} \brac{S(-t)\one_{[0,T]}(t) \Psia(t)}(\tau,n) \|_{L^2_\tau\l^2_n}\notag\\
	&=\Big\|\jb{n}^s\jb{\tau}^b 
	\mathcal{F}_t\big[g_n(t)\big](\tau)
	\Big\|_{L^2_{\tau}\ell^2_n}\,,\label{continue1}
\end{align}
where
	\[g_n(t):=\sum_{j\in\Z^d} \one_{[0,T]}(t) \int_0^t \one_{[0,T]}(t') 
	 e^{-it'|n|^2}\ft{\phi e_j}(n){d}\beta_j(t')\,.\] 
	By the stochastic Fubini theorem 
	(see \cite[Theorem~4.33]{daprato-zab-inf-dim}),  
	we have
	\begin{align*}
	\mathcal{F}_t[g_n(t)](\tau)
	&=\int_{\R}{e^{-it\tau}g_n(t)} dt\\
	&=\sum_{j\in\Z^d}\int_{-\infty}^\infty\one_{[0,T]}(t') e^{-it'|n|^2} \ft{\phi e_j}(n) \int_{t'}^{\infty} \one_{[0,T]}(t)e^{-it\tau}\,{d}t\, {d}\beta_j(t').
	\end{align*}
Since
	\begin{equation}\label{IBP-bound}
		\left|\int_{t'}^{\infty} \one_{[0,T]}(t)e^{-it\tau}\,{d}t\right|\lesssim
		\min\{ T,|\tau|^{-1}\}\,,
	\end{equation}
by Burkholder-Davis-Gundy inequality (Lemma~\ref{BDG}), 
we get
\begin{gather}
\begin{split}
\label{eqn3p5}
	\EE\Big[|\mathcal{F}_t[g_n(t)](\tau)|^{\sigma}\Big]
	&\les \left[
\int_0^T \sum_{j\in\Z^d}	
	\left|\widehat{\phi e_j}(n)
	\int_{t'}^{\infty}\one_{[0,T]}(t)e^{-it\tau}\,dt \right|^2 dt'\right]^{\frac{\sigma}{2}}\\
	&\les \left[ T \sum_{j\in\Z^d}	
	|\widehat{\phi e_j}(n)|^2 
	 \min\{ T^2,|\tau|^{-2}\}\right]^{\frac{\sigma}{2}}\,.
	\end{split}
	\end{gather}
By \eqref{continue1}, \eqref{eqn3p5}, 
and Minkowski inequality,
we get
	\begin{align*}
\norm{\Psia}_{L^{\sigma}(\Omega;X^{s,b}([0,T]))}
	&\le
	\left(
\sum_{n\in\Z^d}\int_{-\infty}^\infty{\jb{n}^{2s}\jb{\tau}^{2b}
	\brac{\EE\left[\left|\mathcal{F}[g_n](\tau)\right|^{\sigma}\right]}^\frac{2}{\sigma}}{\,d\tau}
	\right)^{\frac12}\\
	&\lesssim
	T^{\frac12}
	\left( \sum_{n,j\in\Z^d}\jb{n}^{2s}|\widehat{\phi e_j}(n)|^2\int_{-\infty}^\infty{\jb{\tau}^{2b}\min\{ T^2,|\tau|^{-2}\}}{\,d\tau}\right)^{\frac12}\\
	&\lesssim
	T^{\frac12}\norm{\phi}_{\L^2(L^2;H^s)}
	\left(T^2\intd{|\tau|<1}{}{\tau}+
	\intd{|\tau|\ge 1}{\jb{\tau}^{2b-2}}{\tau}\right)^{\frac12} .
	\end{align*}
 This completes the proof of Lemma~\ref{stoc-conv-est-add}. 
\end{proof}
We now prove that $\Psi$ has a continuous version taking values in $H^s(\T^d)$. This is the content of the next lemma. 

\begin{lemma}[Continuity of the additive noise]
\label{cts-stoc-conv-add}
Let $s\ge 0$, $T>0$, and $2\le \sigma<\infty$.   
Assume that 
$\phi\in \mathcal{L}^2(L^2(\T^d); H^s(\T^d))$. 
Then $\Psia(\cdot)$ belongs to $C([0,T];H^s(\T^d))$ almost surely
and
\begin{equation}
\label{lem3p4:est}
\EE\Bigg[\sup_{t\in[0,T]}\norm{\Psia(t)}_{H^s(\T^d)}^{\sigma}\Bigg] 
\les_{T} \,\norm{\phi}^{\sigma}_{\L^2(L^2(\T^d);H^s(\T^d))}\,.
\end{equation}
\end{lemma}
\begin{proof}
	We fix $\alpha\in \brac{0, \frac12}$ 
	and we write the stochastic convolution as follows: 
	\begin{align}
	\begin{split}\label{factorization}
	\Psia(t)
	&=\frac{\sin(\pi\alpha)}{\pi}\intud{0}{t}{\left[\intud{\mu}{t}{(t-t')^{\alpha-1}(t'-\mu)^{-\alpha}}{t'}\right]S(t-\mu)\phi}{W(\mu)}\\
	&=\frac{\sin(\pi\alpha)}{\pi}\intud{0}{t}{S(t-t')(t-t')^{\alpha-1}
		\intud{0}{t'}{S(t'-\mu)(t'-\mu)^{-\alpha}\phi}{W(\mu)}
	}{t'}\,,
	\end{split}
	\end{align}
	where we used the stochastic Fubini theorem  \cite[Theorem~4.33]{daprato-zab-inf-dim} 
	and the group property of $S(\cdot)$. 
	By Lemma~\ref{fact-meth} and \eqref{factorization} it suffices to show that the process
	\[f(t'):=\intud{0}{t'}{S(t'-\mu)(t'-\mu)^{-\alpha}\phi}{W(\mu)}\]
	satisfies
\begin{equation}
\label{cts-stoc-conv-ref1}
\EE\bigg[\intud{0}{T}{\norm{f(t')}_{H^s_x}^{\sigma}}{t'}\bigg]
 \le C\big(T,\sigma,\norm{\phi}_{\L^2(L^2;H^s)}\big) <\infty\,,
\end{equation}
for some $\sigma>\frac{1}{\alpha}$. 

By Burkholder-Davis-Gundy inequality (Lemma \ref{BDG}), 
	for any $\sigma\ge 2$ and any $t'\in[0,T]$, 
	we get 
	\begin{align*}
	\EE\left[\norm{f(t')}_{H^s_x}^{\sigma}\right]	
	 &\les \left( \int_0^{t'} \|S(t'-\mu)(t'-\mu)^{-\alpha} \phi\|^2_{\L^2(L^2;H^s)} d\mu\right)^{\frac{\sigma}{2}}\\
	 &= \left( \int_0^{t'} (t'-\mu)^{-2\alpha} 
	  \sum_{j\in\Z^d} \|S(t'-\mu)\phi e_j\|^2_{H^s} d\mu \right)^{\frac{\sigma}{2}}\\
	 &\le \|\phi\|_{\L^2(L^2;H^s)}^{\sigma} \left(\frac{T^{1-2\alpha}}{1-2\alpha}\right)^{\frac{\sigma}{2}},
	\end{align*}
	where in the last step we used $2\alpha\in(0,1)$ and the $H^s(\T^d)$-isometry property of $S(t'-\mu)$. 
	Hence
	\begin{align*}
	\textup{LHS of }\eqref{cts-stoc-conv-ref1} = 
	\intud{0}{T}{\EE\left[\norm{f(t')}_{H^s_x}^{\sigma}\right]}{t'}\les \norm{\phi}_{{\L^2(L^2;H^s)}}^{\sigma}{T^{\frac{\sigma}{2}(1-2\alpha)+1}}<\infty\,. 
	\end{align*}
The estimate \eqref{lem3p4:est} follows from \eqref{lem3p2:est}.    
\end{proof}

\subsection{The multiplicative stochastic convolution}\label{subsect:sto-conv-mult}  
The multiplicative stochastic convolution $\Psi=\Psi[u]$ from \eqref{sc:Psim} can be written as
\begin{equation}
\label{sect3:Psiu}
\Psi[u](t) =
\sum_{n\in\Z^d} e_n
   \sum_{j\in\Z^d} \int_0^t e^{i(t-t')|n|^2} \widehat{(u(t') \phi e_j )}(n) d\beta_j(t') .
\end{equation}
Recall that if $s>\frac{d}{2}$, then we have access to the algebra property of $H^s(\T^d)$:
\begin{align}\label{sob-alg}
	\norm{fg}_{H^s(\T^d)}\les \norm{f}_{H^s(\T^d)}\norm{g}_{H^s(\T^d)}
\end{align}
which is an easy consequence of the Cauchy-Schwartz inequality. This simple fact is useful for our analysis in the multiplicative case. On the other hand, \eqref{sob-alg} is not available to us for regularities below $\frac{d}{2}$, 
but we use the following inequalities.
\begin{lemma}
\label{lem3p5}
Let $0<s\le \frac{d}{2}$ and $1\le r < \frac{d}{d-s}$. Then 
\begin{equation}
\label{lem3p5uphij}
\|f u \|_{H^s(\T^d)} \les   \|f\|_{\F L^{s,r}(\T^d)} \|u\|_{H^s(\T^d)} .
\end{equation}
Also, for $s=0$, we have 
\begin{equation}
\label{lem3p5s0}
\|f u \|_{L^2(\T^d)} \les   \|f\|_{\F L^{0,1}(\T^d)} \|u\|_{L^2(\T^d)} .
\end{equation}
\end{lemma}
\begin{proof}
Assume that $0<s\le \frac{d}{2}$ 
and let $n_1$ and $n_2$ denote the spatial frequencies of $f$ and $u$ respectively. By separating the regions $\{|n_1|\gtrsim |n_2|\}$ and $\{|n_1|\ll |n_2|\}$, and then applying Young's inequality, we have 
\begin{align*}
\| f u \|_{H^s(\T^d)} &\les 
\Big\| \big( \widehat{\jb{\nabla}^s f} * \widehat{u}   \big)(n)\Big\|_{\ell_n^2} + 
\Big\|  \big( \widehat{f} * \widehat{\jb{\nabla}^s u}  \big)(n)\Big\|_{\ell_n^2} \\
&\les \|f\|_{\F L^{s,r}} \|\widehat{u}\|_{\ell^p} + \|\widehat{f}\|_{\ell^1} \|u\|_{H^s} \,, 
\end{align*}
where $p$ is chosen such that $\frac1r + \frac1p =\frac32$. 
By H\"{o}lder inequality, for $r'$ and $q$ such that $\frac1r+\frac{1}{r'}=1$ and $\frac1q+\frac12 =\frac1p$,
\begin{align*}
 \|\widehat{f}\|_{\ell^1} &\les \|\jb{n}^{-s}\|_{\ell^{r'}} \|f\|_{\F L^{s,r}}, \\
 \|\widehat{u}\|_{\ell^p} &\les \|\jb{n}^{-s}\|_{\ell^{q}} \|u\|_{H^s} .
\end{align*}
Since $sr'>d$ and $sq>d$ provided that $r<\frac{d}{d-s}$, the conclusion  \eqref{lem3p5uphij} follows.

If $s=0$, \eqref{lem3p5s0} follows easily from Young's inequality: 
\begin{align*}
\| f u \|_{L^2(\T^d)} = \|\widehat{f} * \widehat{u}\|_{\ell^2} \les \|\widehat{f} \|_{\ell^1} \|\widehat{u}\|_{\ell^2} 
 =  \|f\|_{\F L^{0,1}} \|u\|_{L^2} .
\end{align*}
\end{proof}
Given $\phi$ as in Theorem~\ref{lwp:mult}, let us denote 
\begin{equation}
\label{defn:Cphi}
C(\phi):= \norm{\phi}_{\L^2(L^2(\T^d);\F L^{s,r}(\T^d))} <\infty\,,
\end{equation}
for $r=2$ when $s>\frac{d}{2}$, 
for some $r\in\big[1,\frac{d}{d-s}\big)$ when $0<s\le \frac{d}{2}$, and for $r=1$ when $s=0$. 
Recall that if $\phi$ is translation invariant, 
then it is sufficient to assume that $C(\phi)<\infty$ with $r=2$, for all $s\ge 0$. 
We now proceed to prove the following  $X^{s,b}$-estimate of $\Psi[u]$. 

\begin{lemma}\label{stoc-conv-est-mult}
Let $s\ge 0$, $0\le b<\frac{1}{2}$, $T>0$, and $2\le \sigma <\infty$. 
Suppose that 
$\phi$ satisfies the assumptions of Theorem~\ref{lwp:mult}.
Then, for $\Psi[u]$ given by \eqref{sc:Psim} 
we have the estimate
\begin{align}
\EE\left[
\xsbt{\Psi[u]}{s}{b}{[0,T]}^{\sigma}\right]
	&\lesssim (T^2+1)^\frac{\sigma}{2} C(\phi)^{\sigma}\,
\EE\left[\|u\|_{L^2([0,T]; H^s(\T^d))}^{\sigma}\right] 
\label{stoc-conv-est-mult1}\,.
\end{align}
\end{lemma}
\begin{proof}
We first prove \eqref{stoc-conv-est-mult1}. Let $g(t):=\one_{[0,T]}(t)S(-t)\Psim(t)$. 
By the stochastic Fubini theorem  \cite[Theorem~4.33]{daprato-zab-inf-dim},
\begin{align*}
\mathcal{F}_{t,x}(g)(\tau,n)
&=\intd{\R}{
	e^{-it\tau}\one_{[0,T]}(t)\sum_{j\in\Z^d}\intud{0}{t}{
		e^{-it'n^2} (\widehat{u(t')\phi e_j})(n)
	}{\beta_j(t')}
}{t}\\
&=\sum_{j\in\Z^d}\intud{0}{T}{
	\intud{t'}{\infty}{\one_{[0,T]}(t)
		e^{-it\tau}e^{-it'n^2}(\widehat{u(t')\phi e_j})(n)
	}{t}
}{\beta_j(t')}\,.
\end{align*}
Then by \eqref{Xsblocalsim} and the assumption $0\le b<\frac{1}{2}$, 
the Burkholder-Davis-Gundy inequality (Lemma~\ref{BDG}),  
and \eqref{IBP-bound}, we have 
\begin{align*}
\textup{LHS of } \eqref{stoc-conv-est-mult1} 
&\sim \EE\left[\norm{\jb{n}^s\jb{\tau}^b
 \mathcal{F}[g](n,\tau)}_{L^2_{\tau}\ell^2_n}^{\sigma}\right]\\
	&\hspace{-1cm}\lesssim 
	\EE\left[
	\brac{
		\sum_{j,n\in\Z^d}
		\intd{\R}{
			\intud{0}{T}{
				\jb{n}^{2s}\jb{\tau}^{2b}
				\left|\intud{t'}{\infty}{\one_{[0,T]}(t)e^{-it\tau}}{t}\right|^2
				\left| (\widehat{u(t')\phi e_j})(n)\right|^2
			}{t'}
		}{\tau}
	}^{\frac{\sigma}{2}}\right]\\
&\hspace{-1cm}\lesssim 
(T^2+1)^\frac{\sigma}{2}\,\EE\left[
\brac{
		\intud{0}{T}{ \sum_{j,n\in\Z^d}
			\jb{n}^{2s}\left| (\widehat{u(t')\phi e_j})(n)\right|^2
		}{t'}
}^{\frac{\sigma}{2}}\right] \,.
\end{align*}
If $s>\frac{d}{2}$, we apply the algebra property of $H^s(\T^d)$ to get
\begin{equation*}
\|u(t') \phi e_j\|_{\ell^2_j H^s} \les {\|\phi\|_{\L^2(L^2;H^s)}} \|u(t')\|_{H^s}. 
\end{equation*}
If $0\le s\le \frac{d}{2}$, we have 
\begin{equation}
\|u(t') \phi e_j\|_{\ell^2_j H^s} \les C(\phi) \|u(t')\|_{H^s}. 
\end{equation}
and thus \eqref{stoc-conv-est-mult1} follows. 
\end{proof}

Next, we prove the continuity of $\Psi[u](t)$ in the same way as in Lemma~\ref{cts-stoc-conv-add}, 
i.e. by using Lemma~\ref{fact-meth}.

\begin{lemma}[Continuity of the multiplicative noise]
\label{cts-stoc-conv-mult}
Let $T>0$, $s\ge 0$, $0\le b<\frac{1}{2}$, 
and $2\le \sigma <\infty$. 
Suppose that $u\in L^{\sigma}\big(\Omega;X^{s,b}([0,T])\big)$ 
 and that
$\phi$ satisfies the assumptions of Theorem~\ref{lwp:mult}.
	Then $\Psi[u](\cdot)$ given by \eqref{sect3:Psiu} 
	belongs to $C([0,T];H^s(\T^d))$ almost surely. 
	Moreover, 
\begin{equation}
\label{stoc-conv-est-mult2}
\EE\left[\sup_{t\in[0,T]}\norm{\Psi[u](t)}_{H^s(\T^d)}^{\sigma} \right] \les 
C(\phi)^{\sigma} \, 
 \EE\left[\|u\|_{X^{s,b}([0,T])}^{\sigma}\right]\,.
\end{equation}
\end{lemma}
\begin{proof}
Applying the same factorisation procedure as in the proof of Lemma~\ref{cts-stoc-conv-add} 
reduces the problem to proving that the process 
	\[f(t'):=\intud{0}{t'}{(t'-\mu)^{-\alpha}S(t'-\mu)\big[u(\mu)\phi\big]}{W(\mu)}\]
satisfies
	\begin{equation}\label{cts-stoc-conv-mult-ref1}
	\EE\left[\intud{0}{T}{\norm{f(t')}_{H^s_x}^{\sigma}}{t'}\right]\le C'\brac{T,\sigma, C(\phi)} <\infty \,
	\end{equation}
for some $0<\alpha<1$ satisfying $\alpha>\frac{1}{\sigma}$. 
By the Burkholder-Davis-Gundy inequality (Lemma~\ref{BDG}) and Lemma~\ref{lem3p5}, 
we have
	\begin{align*}
	\EE\left[\norm{f(t')}_{H^s_x}^{\sigma}\right]	
	 &\les \EE\left[ \left( \int_0^{t'} \|(t'-\mu)^{-\alpha} S(t'-\mu) [u(\mu)\phi]\|^2_{\L^2(L^2;H^s)} d\mu\right)^{\frac{\sigma}{2}}\right]\\
	 &= \EE\left[  \left( \int_0^{t'} (t'-\mu)^{-2\alpha} 
	  \sum_{j\in\Z^d} \|S(t'-\mu)u(\mu)\phi e_j\|^2_{H^s} d\mu \right)^{\frac{\sigma}{2}}\right]\\
	 &\les \EE\left[\left(  \sum_{j\in\Z^d} \|\phi e_j\|^2_{\F L^{s,r}}  \int_0^{T} (t'-\mu)^{-2\alpha} 
	 \|u(\mu)\|^2_{H^s} d\mu \right)^{\frac{\sigma}{2}}\right]\,.
	\end{align*}
Then, by Fubini theorem and Minkowski inequality, we obtain 
\begin{align*}
\EE\left[\int_0^T \norm{f(t')}_{H^s_x}^{\sigma} dt' \right]	 
 &= \Big\|\, \|f\|_{H^s_x}\Big\|^{\sigma}_{L^\sigma(\Omega; L^\sigma_{t'}[0,T])}\\
 &\les  C(\phi)^{\sigma}\, \bigg\|\, 
 \Big\| (t'-\mu)^{-\alpha} \|u(\mu)\|_{H^s_x}\Big\|_{L^2_{\mu}(0,T])}
 \bigg\|^{\sigma}_{L^\sigma(\Omega; L^\sigma_{t'}[0,T])}\\
 &\le C(\phi)^{\sigma}\, \EE \Bigg[ \bigg\|\, 
 \Big\| (t'-\mu)^{-\alpha} \|u(\mu)\|_{H^s_x}\Big\|_{L^{\sigma}_{t'}(0,T])}
 \bigg\|^{\sigma}_{L_{\mu}^{2}([0,T])}\Bigg]\\
 &\les C(\phi)^{\sigma}\, \EE \Bigg[ \Bigg(\int_0^T (T-\mu)^{2(\frac{1}{\sigma}-\alpha)} \|u(\mu)\|_{H^s_x}^2d\mu 
 \Bigg)^{\frac{\sigma}{2}} \Bigg]
\end{align*}
By H\"{o}lder and Sobolev inequalities and \eqref{Xsblocalsim}, 
we have 
	\begin{align*}
	\Bigg( \int_0^T (T-\mu)^{2(\frac{1}{\sigma}-\alpha)} \|u(\mu)\|^2_{H_x^s} d\mu \Bigg)^{\frac12}
&\le 
	\Big\| (T-\mu)^{\frac{1}{\sigma} - \alpha} \Big\|_{L_{\mu}^{\frac{4}{1+2b}}([0,T])}
	 \Big\| \|u(\mu)\|_{H^s_x}\Big\|_{L_{\mu}^{\frac{4}{1-2b}}([0,T])}\\
	 &\les T^{1+\frac{4}{1+2b}(\frac{1}{\sigma}-\alpha)}  
	  \Big\|\one_{[0,T]}(\mu) \|S(-\mu) u(\mu)\|_{H^s_x} \Big\|_{L_{\mu}^{\frac{4}{1-2b}}}\,.
	 \end{align*}
There exists $\alpha=\alpha(\sigma):=\frac{1}{\sigma}+\frac14$ for which we have 
\begin{align*}
\EE\left[\int_0^T \norm{f(t')}_{H^s_x}^{\sigma} dt' \right]	&\les 
 \EE\Big[ T^{\frac{2b\sigma}{1+2b}} \|u\|^{\sigma}_{X^{s,b}([0,T])} \Big] <\infty \,.
\end{align*} 
\end{proof}

\section{Local well-posedness}
\label{sect:LWP}
\subsection{SNLS with additive noise}
\label{sect:LWPa}
In this subsection, we prove  Theorem~\ref{lwp:add}. 
Let $b=b(k)=\frac{1}{2}-$ be given by Lemma \ref{trilin} (in the case $d=k=1$) 
or by Lemma \ref{multilin} (in the case $dk\ge 2$). 
By Lemma \ref{stoc-conv-est-add}, for any $T>0$, 
there is an event $\Omega'$ of full probability such that 
the stochastic convolution $\Psia$ has finite $X^{s,b}([0,T])$-norm on $\Omega'$.  
	
	Now fix $\omega\in \Omega'$ and $u_0\in H^s(\T^d)$. 
	Consider the ball
\[B_R:=\big\{u\in X^{s,b}([0,T]):\|u\|_{X^{s,b}([0,T])}\le R\big\}\]
	where $0<T<1$ and $R>0$ are to be determined later. 
We aim to  show that the operator $\Lambda$ given by 
$$ \Lambda u(t) 
= S(t) u_0 \pm i\int_0^t S(t-t') \big(|u|^{2k}u\big)(t')dt' - i \Psi(t)\ ,\ t\geq0,\, $$
where $\Psi$ is the additive stochastic convolution given by \eqref{Psia}, is a contraction on $B_R$. 
To this end, it remains to estimate the $X^{s,b}([0,T])$-norm of 
\begin{equation*}
D(u):=\intud{0}{t}{S(t-t')\big(|u|^{2k}u\big)(t')}{t'} \,.
\end{equation*}
For any $\delta>0$ sufficiently small (such that $b+\delta<\frac12$),  
by Lemma \ref{xsb-time-loc}  and  \eqref{Xsblocalsim}:
\begin{align*}
	\left\|D(u)\right\|_{X^{s,b}([0,T])} 
	\lesssim T^{\delta} \left\|D(u)\right\|_{X^{s,b+\delta}([0,T])} 
	\lesssim  T^\delta\left\|\one_{[0,T]}(t) D(u)(t) \right\|_{X^{s,\frac{1}{2}+\delta}}.
\end{align*}
Let $\eta$ be a smooth cut-off function, supported on $[-1, T+1]$, with $\eta(t)=1$ for all $t\in [0,T]$. 
For any $w\in X^{s,-\frac12+\delta}$  that  agrees with 
$|u|^{2k}u$ on $[0,T]$, 
by Lemma  \ref{lem:linests}, we obtain 
\begin{align}
\label{sect4p1estDu1}
 \left\|\one_{[0,T]}(t) D(u)(t)\right\|_{X^{s,\frac12+\delta}} & 
  \lesssim  
	\left\|\eta(t) \int_0^t S(t-t') w(t') dt' \right\|_{X^{s,\frac{1}{2}+\delta}}
\lesssim  \|w\|_{X^{s,-\frac{1}{2}+\delta}}
\end{align}
Then after taking the infimum over all such $w$, we use Lemma \ref{trilin} or \ref{multilin} and we get 
\begin{align}
\label{sect4p1estDu2}
\left\|D(u)\right\|_{X^{s,b}([0,T])} 
	&\lesssim T^\delta\| (u \overline{u})^k u\|_{X^{s,-\frac12+\delta}([0,T])}
\lesssim T^\delta\norm{u}^{2k+1}_{X^{s,b}([0,T])}.
\end{align}
It follows that 
\begin{equation}\label{random4}
	\norm{\Lambda u}_{X^{s,b}([0,T])}\le c \norm{u_0}_{H^s_x}
	 +c T^{\delta} \norm{u}^{2k+1}_{X^{s,b}([0,T])} +\xsbt{\Psia(t)}{s}{b}{[0,T]}, 
\end{equation}
for some $c>0$. 
Similarly, we obtain
\begin{equation}\label{random5}
\norm{\Lambda u-\Lambda v}_{X^{s,b}([0,T])}\le c T^{\delta} \brac{\norm{u}^{2k}_{X^{s,b}([0,T])}
 +\norm{v}^{2k}_{X^{s,b}([0,T])}}\norm{u-v}_{X^{s,b}([0,T])}.
\end{equation}
Let $R:= 2c \norm{u_0}_{H_x^s}+2\xsbt{\Psia(t)}{s}{b}{[0,T]}$. 
From \eqref{random4} and \eqref{random5}, 
we see that $\Lambda$ is a contraction from $B_R$ to $B_R$ provided
\begin{equation}\label{random6}
cT^\delta R^{2k+1}\le\frac{1}{2}R \ \mbox{ and } \  c T^{\delta} \brac{2R^{2k}}\le\frac{1}{2}\,.
\end{equation}
This is always possible if we choose $T\ll 1$ sufficiently small. 
This shows the existence of a unique solution $u\in X^{s,b}([0,T])$ to \eqref{SNLS-mild} on $\Omega'$.

Finally, we check that $u\in C([0,T]; H^s)$ on the set of full probability $\Omega''\cap \Omega'$, 
where $\Omega''$ is given by Lemma \ref{cts-stoc-conv-add}, 
that is $\Psi\in C([0,T];H^s)$ on $\Omega''$. 
By \eqref{Xsblocalsim}, \eqref{sect4p1estDu1}  and Lemma \ref{trilin} or \ref{multilin}, we also get
\begin{equation}
\xsbt{D(u)}{s}{\frac{1}{2}+\delta}{[0,T]}\lesssim \left\|\one_{[0,T]}(t) D(u)(t)\right\|_{X^{s,\frac12+\delta}} 
 \les \xsbt{u}{s}{b}{[0,T]}^{2k+1} .
\end{equation}
By the embedding $X^{s,\frac{1}{2}+\delta}([0,T]) \hookrightarrow C([0,T];H^s(\T^d))$, 
we have $D(u)\in C([0,T];H^s(\T^d))$. 
Since the linear term $S(t)u_0$ also belongs to $C([0,T];H^s(\T^d))$, 
we conclude that 
$$u=\Lambda u\in C\big([0,T];H^s(\T^d)\big) \text{ on } \Omega''\cap \Omega'.$$

\begin{remark}
From \eqref{random6}, we obtain the time of existence
\begin{equation}
\label{time-of-existence-a}
T_{\text{max}}:=\max\bigg\{\tilde{T}>0: 
 \tilde{T}\le c\Big(\norm{u_0}_{H^s}+\norm{\Psi}_{X^{s,b}([0,\tilde{T}])}\Big)^{-\theta}\bigg\}\,,
\end{equation}
where $\theta=\frac{2k}{\delta}$. Note that \eqref{time-of-existence-a} will be useful in our global argument.
\end{remark}

\subsection{SNLS with multiplicative noise}\label{sect:LWPm}
In this subsection, we prove Theorem~\ref{lwp:mult}. 
Following \cite{debouard-debussche-kdv}, 
we use a truncated version of \eqref{SNLS-mild}. 
The main idea is to apply an appropriate cut-off function on the nonlinearity to obtain a family of truncated SNLS, and then prove global well-posedness of these truncated equations. 
Since solutions started with the same initial data coincide up to suitable stopping times, 
we obtain a solution to  the original SNLS in the limit.

Let $\eta:\R\to[0,1]$ be a smooth cut-off function such that $\eta\equiv 1$ on $[0,1]$ and $\eta\equiv 0$ outside $[-1,2]$. 
Set $\eta\sub{R}:=\eta\brac{\frac{\cdot}{R}}$ and  consider the equation
\begin{equation}\label{SNLSm-R}
i\partial_tu\sub{R} -\Delta u\sub{R} \pm \eta_R\big(\xsbt{u_R}{s}{b}{[0,t]}\big)^{2k+1}|{u}\sub{R}|^{2k} u\sub{R} = u_R\cdot\phi\xi\,,
\end{equation}
with initial data $u_R|_{t=0} = u_0$. 
Its mild formulation is $u_R=\Lambda_R u_R$, where $\Lambda_R$ is given by
\begin{align}
\label{SNLS-R-mild}
\Lambda\sub{R} u_R&:= S(t)u_0
\pm i\intud{0}{t}{S(t-t')\eta_R\brac{\xsbt{u_R}{s}{b}{[0,t']}}^{2k+1}|u\sub{R}|^{2k}u\sub{R}(t')}{t'}-i\Psim[u\sub{R}](t)\,.
\end{align}
The key ingredient for Theorem~\ref{lwp:mult} is  the following proposition.

\begin{proposition}[Global well-posedness for \eqref{SNLSm-R}]\label{gwp-mult-R}
Let $s>s_{\text{crit}}$, $s\ge 0$, and $T, R>0$. Suppose that $\phi$ is as in Theorem~\ref{lwp:mult}.  
Given $u_0\in H^s(\T^d)$, there exists $b=\frac{1}{2}-$ 
and a unique adapted process
$$u\sub{R}\in L^2\Big(\Omega;C\big([0,T];H^s(\T^d)\big)\cap X^{s,b}([0,T])\Big)$$
solving \eqref{SNLSm-R} on $[0,T]$.
\end{proposition}

Before proving this result, we state and prove the following lemma. 

\begin{lemma}[Boundedness of cut-off]\label{cutoff-bdd}
	Let $s\ge0$, $b\in[0,\frac{1}{2})$, $R>0$ and $T>0$. 
	There exist constants $C_1, C_2(R)>0$ such that
\begin{align}
	\xsbt{\eta_{R}
		\brac{
			\xsbt{u}{s}{b}{[0,t]}
		}
		u(t)
	}{s}{b}{[0,T]}&\le \min\left\{C_1\xsbt{u}{s}{b}{[0,T]}, C_2(R)\right\}\,;\label{cutoff-bb1}
\end{align}
\begin{align}
\xsbt{\eta_{R}\brac{\xsbt{u}{s}{b}{[0,t]}}u(t)
-\eta_{R}\brac{\xsbt{v}{s}{b}{[0,t]}}
v(t)
}{s}{b}{[0,T]}&\le C_2(R)\xsbt{u-v}{s}{b}{[0,T]}\,.\label{cutoff-bb2}
\end{align}
\end{lemma}
\begin{proof}
We first prove \eqref{cutoff-bb1}. Let $w(t,n)=\F_x[S(-t)u(t)](n)$, $\kappa\sub{R}(t)=\eta\sub{R}
\brac{\xsbt{u}{s}{b}{[0,t]}}$ and
\begin{equation}\label{tauR}
	\tau\sub{R}:=\inf\left\{t\ge 0: \xsbt{u}{s}{b}{[0,t]}\ge 2R\right\}\,.
\end{equation}
Then $\kappa\sub{R}(t)=0$ when $t>\tau\sub{R}$. By \eqref{Xsblocalsim} and \eqref{Xsb-interact-rep}, 
\begin{align}
\xsbt{\kappa\sub{R}(t)
	u(t)
}{s}{b}{[0,T]}^2
&\sim
\xsb{\one_{[0,T\wedge\tau\sub{R}]}\kappa\sub{R}(t)u(t)}{s}{b}^2
\sim \xsbt{\kappa\sub{R}(t)u(t)
}{s}{b}{[0,T\wedge\tau\sub{R}]}^2 \notag\\
&\sim\sum_{n\in\Z^d}\jb{n}^{2s}\norm{\kappa\sub{R}(t)
		w(t,n)}^2_{H^b(0,T\wedge\tau\sub{R})}. 
		\label{random7}
\end{align}
We now estimate the $H^b(0,T\wedge\tau\sub{R})$-norm, for which we use the following characterization (see for example \cite{tartar-sobolev}):
\begin{align}
	\norm{f}_{H^b(a_1,a_2)}^2\sim\norm{f}_{L^2(a_1,a_2)}^2+
	\intud{a_1}{a_2}{
		\intud{a_1}{a_2}{\frac{|f(x)-f(y)|^2}{|x-y|^{1+2b}}}{x}
	}{y}\ , \quad 0<b<1. 
	\label{sobolev-char}
\end{align}
For the inhomogeneous contribution (i.e. coming from the $L^2$-norm above), 
we have
\begin{align*}
\sum_{n\in\Z^d}\jb{n}^{2s}\norm{\kappa\sub{R}(t)w(t,n)}^2_{L^2_{t}(0,T\wedge\tau\sub{R})}
 &\le \min\left\{\xsbt{u}{s}{b}{[0,\tau\sub{R}]}^2, \xsbt{u}{s}{b}{[0,T]}^2\right\}\\
& \le \min\left\{\brac{2R}^2, \xsbt{u}{s}{b}{[0,T]}^2\right\}.
\end{align*}
The remaining part of \eqref{random7} needs a bit more work. Fix $n\in\Z^d$, then
\begin{align*}
\longeqn
\intud{0}{T\wedge\tau\sub{R}}{
	\intud{0}{T\wedge\tau\sub{R}}{\frac{|\kappa\sub{R}(t)w(t,n)
			-\kappa\sub{R}(t')w(t',n)|^2
		}{|t-t'|^{1+2b}}}{t'}
}{t}\\
&\les\intud{0}{T\wedge\tau\sub{R}}{
	\intud{0}{t}{\frac{|\kappa\sub{R}(t)(w(t,n)
			-w(t',n))|^2
		}{|t-t'|^{1+2b}}}{t'}
}{t}\\
&\quad\quad\quad+\intud{0}{T\wedge\tau\sub{R}}{
	\intud{0}{t}{\frac{|(\kappa\sub{R}(t)-\kappa\sub{R}(t'))w(t',n)|^2
		}{|t-t'|^{1+2b}}}{t'}
}{t}\\
&=: \Irm(n)+\IIrm(n)\,.
\end{align*}
It is clear that  
$$\Irm(n)\lesssim \min\left\{\norm{w(n)}_{H^b((0,\tau\sub{R}))}^2, \norm{w(n)}_{H^b((0,T))}^2 \right\}\,,$$
and hence
\[\sum_{n \in\Z^d}\Irm(n)\les \min\left\{\brac{2R}^2, \xsbt{u}{s}{b}{[0,T]}^2\right\}.\] 
For $\IIrm(n)$, the mean value theorem infers that
\begin{align*}
\left|\kappa_R(t)-\kappa_R(t')\right|^2
&\lesssim\frac{\left(\xsbt{u}{s}{b}{[0,t]}-\xsbt{u}{s}{b}{[0,t']}\right)^2}{R^2}\left(\sup_{r\in\R}\eta'(r)\right)^2\\
&\les \frac{\xsb{\one_{[t',t]}u}{s}{b}^2}{R^2}\\
&\lesssim \frac{1}{R^2}\sum_{n'\in\Z^d}\jb{n'}^{2s}\|{w(\cdot,n')}\|^2_{H^b(t',t)} .
\end{align*}
Again, we split $\|{w(\cdot,n')}\|^2_{H^b(t',t)}$ using \eqref{sobolev-char} 
into the inhomogeneous contribution 
(the $L^2$-norm squared part) and the homogeneous contribution (the second term of \eqref{sobolev-char}). 
We  control here only the homogeneous contributions for $\IIrm(n)$ as the inhomogeneous contributions are easier. 
The homogeneous part of  $\IIrm(n)$ is controlled by 
\begin{align}
\longeqn
\frac{1}{R^2}\sum_{n'\in\Z^d}\jb{n'}^{2s}
\intud{0}{T\wedge\tau\sub{R}}{
	\intud{0}{t}{
		\intud{t'}{t}{
			\intud{t'}{\lambda}{
				\frac{|w(t',n)|^2}{|t-t'|^{1+2b}}\cdot\frac{|w(\lambda,n')-w(\lambda',n')|^2}{|\lambda-\lambda'|^{1+2b}}				
			}{\lambda'}
		}{\lambda}
	}{t'}
}{t}\\
&=
\frac{1}{R^2}\sum_{n'\in\Z^d}\jb{n'}^{2s}
\int_0^{T\wedge\tau\sub{R}}\,\int^\lambda_0\,\int^{\lambda'}_0\,\brac{\intud{\lambda}{T\wedge\tau\sub{R}}{\frac{1}{|t-t'|^{1+2b}}}{t}}|w(t',n)|^2\notag\\
&\hspace{5.4cm}\times\frac{|w(\lambda,n')-w(\lambda',n')|^2}{|\lambda-\lambda'|^{1+2b}}\,dt'\,d\lambda'\,d\lambda\,,\label{random8}
\end{align}
where we used $0\le t'\le \lambda'\le \lambda\le t\le T\wedge \tau\sub{R}$ to switch the integrals. Now, the integral with respect to $t$ is equal to $|T\wedge\tau\sub{R}-t'|^{-2b}-|\lambda-t'|^{-2b}$, which is bounded by
\[|T\wedge\tau\sub{R}-t'|^{-2b}\le |\lambda'-t'|^{-2b}\,.\]
Thus \eqref{random8} is controlled by
\begin{align}
\frac{1}{R^2}\sum_{n'\in\Z^d}\jb{n'}^{2s}
&\int_{0}^{T\wedge\tau\sub{R}}\,
\int_{0}^{\lambda}\,
\brac{\intud{0}{\lambda'}{
		|\lambda'-t'|^{-2b}
		|w(t',n)|^2
	}{t'}}\notag\\
&\hspace{1.8cm}\times\frac{|w(\lambda,n')-w(\lambda',n')|^2}{|\lambda-\lambda'|^{1+2b}}
{\,d\lambda'}
{\,d\lambda}\,.\label{random9}
\end{align}
Since $b\in\left[0,\frac{1}{2}\right)$, 
by Hardy's inequality (see for example \cite[Lemma A.2]{tao-book}) the $t'$-integral is 
$\les \norm{w(\cdot,n)}_{H^b(0,\lambda')}^2\le \norm{w(\cdot,n)}_{H^b(0,T\wedge \tau\sub{R})}^2$. After multiplying by $\jb{n}^{2s}$ and summing over $n\in\Z^d$, we see that (\ref{random9}) is controlled by
\begin{align*}
&\frac{1}{R^2}\sum_{n,n'\in\Z^d}\jb{n}^{2s}\jb{n'}^{2s}\norm{w(\cdot,n)}_{H^b(0,T\wedge\tau\sub{R})}^2\norm{w(\cdot,n)}_{H^b_\lambda(0,T\wedge\tau\sub{R})}^2\\
&\hspace{2cm}\lesssim \frac{1}{R^2}\xsbt{u}{s}{b}{[0,T\wedge\tau\sub{R}]}^2\xsbt{u}{s}{b}{[0,T\wedge\tau\sub{R}]}^2\\
&\hspace{2cm}\le \min\left\{4\xsbt{u}{s}{b}{[0,T]}^2,16R^2\right\}\,.
\end{align*}

We now prove \eqref{cutoff-bb2}. Let $\tau^u_R$ and $\tau^v_R$ be defined as in \eqref{tauR}. Assume without loss of generality that $\tau^u_R\le\tau^v_R$. We decompose
\begin{align*}
	\text{LHS of }\eqref{cutoff-bb2}
	&\les
\xsbt{\brac{\eta_{R}\brac{\xsbt{u}{s}{b}{[0,t]}}
-\eta_{R}\brac{\xsbt{v}{s}{b}{[0,t]}}}
v(t)}{s}{b}{[0,T]}\\
	&\quad\quad\quad+\xsbt{\eta_{R}\brac{\xsbt{u}{s}{b}{[0,t]}}\brac{u(t)-v(t)}
	}{s}{b}{[0,T]}\\
&=:A+B\,.
\end{align*}
By the mean value theorem,
\begin{align*}
A&= \xsbt{\brac{\eta_{R}\brac{\xsbt{u}{s}{b}{[0,t]}}
		-\eta_{R}\brac{\xsbt{v}{s}{b}{[0,t]}}}
	v(t)}{s}{b}{[0,T\wedge\tau_R^v]}\\
&\les \frac{1}{R}\xsbt{v}{s}{b}{[0,T\wedge\tau^v_R]}\xsbt{u-v}{s}{b}{[0,T]}\\
&\les \xsbt{u-v}{s}{b}{[0,T]}\,.
\end{align*}
For $B$, one runs through the same argument as for \eqref{cutoff-bb1} 
but with $w(t,n)$ replaced by $\F_x\big[S(-t)\big(u(t)-v(t)\big)\big](n)$, which yields
\[B\les C(R)\xsbt{u-v}{s}{b}{[0,T]}\,.\qedhere\]
\end{proof}
We now conclude the proof of Proposition \ref{gwp-mult-R}.
\begin{proof}[Proof of Proposition \ref{gwp-mult-R}]
Let $T,R>0$ and let $E\sub{T}:=L_{\textup{ad}}^2\brac{\Omega;X^{s,b}([0,T])}$ be the space of adapted processes in $L^2\brac{\Omega;X^{s,b}([0,T])}$. 
We solve the fixed point problem \eqref{SNLS-R-mild} in $E_T$. Arguing as in the additive case, and using Lemmata  \ref{cutoff-bdd} and \ref{stoc-conv-est-mult}, we have
\begin{align*}
\norm{\Lambda\sub{R} u}_{E_T}&\le C_1 \norm{u_0}_{H^s}+C_2(R)T^{\delta} +C_3T^{b}\norm{u}_{E_T}\,;\\
\norm{\Lambda\sub{R} u-\Lambda\sub{R} v}_{E_T}&\le C_4(R) T^\delta\norm{u-v}_{E_T}+C_5T^b\norm{u-v}_{E_T}\,.
\end{align*}
Therefore, $\Lambda\sub{R}$ is a contraction from $E_T$ to $E_T$ provided we choose $T=T(R)$ sufficiently small. Thus there exists a unique solution $u_R\in E_T$. Note that $T$ does not depend on $\norm{u_0}_{H^s}$, hence we may iterate this argument to extend $u_R(t)$ to all $t\in [0,\infty)$.

Finally, to see that 
$u\sub{R}\in F_T:=L^2\big(\Omega; C([0,T];H^s(\T^d))\big)$, we first note that since $u_R\in E_T$, Lemma \ref{cts-stoc-conv-mult} infers that $\Psi[u_R]\in F_T$. Then, by similar argument as in the end of Subsection \ref{sect:LWPa}, we have that $D(u_R)\in L^2(\Omega;X^{s,\frac{1}{2}+}\big([0,T]\big))$, where
\[D(u_R)(t):=\int_0^t S(t-t')\big(|u_R|^{2k}u_R\big)\,dt'\,.\]
Since $L^2\big(\Omega; X^{s,\frac12+}([0,T])\big)\hookrightarrow F_T$, we have $D(u_R)\in F_T$. 
Also, it is clear that $S(t)u_0\in F_T$. Hence $u_R\in F_T$.
\end{proof}

\begin{proof}[Proof of Theorem~\ref{lwp:mult}]
Let
\begin{equation}
\label{defn:tauR4p17}
	\tau\sub{R}:=\inf\big\{t>0:\xsbt{u\sub{R}}{s}{b}{[0,t]}\ge R\big\}.
\end{equation}
Then, $\eta\sub{R}(\xsbt{u\sub{R}}{s}{b}{[0,t]})=1$ if and only if $t\le\tau\sub{R}$. 
Hence $u\sub{R}$ is a solution of \eqref{SNLS-mild} on $[0,\tau\sub{R}]$. 
For any $\delta>0$, we have  $u\sub{R}(t)=u\sub{R+\delta}(t)$ whenever $t\in[0,\tau\sub{R}]$. Consequently, $\tau\sub{R}$ is increasing in $R$. Indeed, if $\tau\sub{R}>\tau\sub{R+\delta}$ for some $R>0$ and some $\delta>0$, then for $t\in[\tau\sub{R+\delta},\tau\sub{R}]$, 
we have $\eta\sub{R+\delta}\big(\xsbt{u\sub{R+\delta}}{s}{b}{[0,t]}\big)<1$ which implies that $u\sub{R}(t)\ne u\sub{R+\delta}(t)$, a contradiction. Therefore, 
\begin{equation}
	\tau^*:=\lim_{R\to\infty}\tau\sub{R}
\end{equation}
is a well-defined stopping time that is either positive or infinite almost surely. By defining $u(t):=u\sub{R}(t)$ for each $t\in[0,\tau\sub{R}]$, we see that $u$ is a solution of \eqref{SNLS-mild} on $[0,\tau^*)$ almost surely.
\end{proof}

\section{Global well-posedness}
\label{sect:GWP}
In this section, we prove Theorems \ref{gwp:add} and \ref{gwp:mult}. 
Recall that the \emph{mass} and \emph{energy} of a solution $u(t)$ 
of the defocusing 
\eqref{SNLS} are  given respectively by
\begin{align}
M(u(t))&=\intd{\T^d}{\frac12 |u(t,x)|^2}{x} , \label{mass}\\
E(u(t))&=\int_{\T^d} \frac{1}{2} |\nabla u(t,x)|^2 +\frac{1}{2(k+1)} |u(t,x)|^{2(k+1)} dx . \label{energy}
\end{align}
It is well-known that these are conserved quantities for (smooth enough) 
solutions of the deterministic NLS equation.

For SNLS, we prove probabilistic a priori control 
as per Propositions~\ref{apriori-est-M-a} 
and \ref{apriori-est-M-m} below. 
To this purpose, 
the idea is to compute the stochastic differentials of \eqref{mass} and \eqref{energy} 
and use the stochastic equation for  $u$. 
We  work with the following frequency truncated version of \eqref{SNLS}: 
\begin{equation}
\label{SNLS_N}
\begin{cases}
i\partial_t u^N - \Delta u^N  \pm P_{\leq N} |u^N|^{2k}u^N = F(u^N,\phi^N dW^N),\\
u^N|_{t=0} = P_{\le N} u_0=: u_0^N
\end{cases}
\end{equation}
where $P_{\leq N}$ is the Littlewood-Paley projection onto the frequency set 
$\{n\in\Z^d:|n|\le N\}$, 
$$\phi^N:= P_{\leq N}\circ \phi\ \text{ and }\ W^N(t):= \sum_{|n|\leq N} \beta_n(t) e_n.$$

By repeating the arguments in Section~\ref{sect:LWP}, 
one obtains local well-posedness for \eqref{SNLS_N} with initial data $P_{\le N}u_0$ at least with the same time of existence as for the untruncated SNLS.

\subsection{SNLS with additive noise}
We treat the additive SNLS in this subsection. We first prove probabilistic a priori bounds on \eqref{mass} and \eqref{energy} of a solution $u^N$ of the truncated equation.

\begin{proposition}\label{apriori-est-M-a}
	Let $m\in\NN$, $T_0>0$,  $\phi\in\L^2(L^2(\T^d);L^2(\T^d))$, and $F(u,\phi\xi)=\phi\xi$. 
	Suppose that  $u^N(t)$ is a solution to \eqref{SNLS_N} for $t\in[0,T]$,  
	for some stopping time $T\in[0,T_0]$. 
	Then  there exists a constant 
	$C_1=C_1(m,M(u_0), T_0, \|\phi\|_{\L^2(L^2;L^2)})>0$ such that
	\begin{align}
	\label{l2-est-a}
	\EE\left[\sup_{0\le t\le T} M(u^N(t))^m \right] 
	\le C_1\,.
	\end{align}
	Furthermore, if \eqref{SNLS_N} is defocusing,
	 there  exists $C_2=C_2(m,E(u_0), T_0, \|\phi\|_{\L^2(L^2;H^1)})>0$ such that	
	\begin{align}
	\label{h1-est-a}\EE\left[\sup_{0\le t\le T} E(u^N(t))^m \right] 
	&\le C_2\,.
	\end{align}
	The constants $C_1$ and $C_2$ are independent of $N$.
\end{proposition}
\begin{proof} 
By applying It\^o's Lemma, we have
\begin{align}
	M(u^N(t))^m&=M(u_0^N)^m\notag\\
	&\phantom{=}+
	m\,\Im\brac{\sum_{|j|\le N}\int_{0}^{t}{M(u^N(t'))^{m-1}\int_{\T^d}\cj{u^N(t')}\phi^Ne_j\,dx}{\,d\beta_j(t')}}\label{mass-a1}\\
	&\phantom{=}+m(m-1)\sum_{|j|\le N}\int_{0}^{t}{M(u^N(t'))^{m-2}\left|\int_{\T^d}u^N(t')\phi^Ne_j\,dx\right|^2}{\,dt'}\,.\label{mass-a2}\\
	&\phantom{=}+m\norm{\phi^N}^2_{\L^2(L^2;L^2)}\int_0^t{M(u^N(t'))^{m-1}}{\,dt'}\label{mass-a3},
	\end{align}
the last term being the It\^o correction term. 
We first control \eqref{mass-a1}. 
By Burkholder-Davis-Gundy inequality (Lemma~\ref{BDG}),  
H\"older and Young  inequalities, we get
\begin{align*}
\EE\left[\supT\eqref{mass-a1}\right] 
&\lesssim_m \EE\left[  \left\{\sum_{|j|\leq N}\int_0^T M(u^N(t'))^{2(m-1)} \|u^N(t')\|_{L^2}^2 \|\phi^Ne_j\|_{L^2}^2 dt' \right\}^\frac12 \right]\\
&\lesssim {\|\phi^N\|}_{\L^2(L^2;L^2)} \, \EE\left[ \left\{ \int_0^T M(u^N(t))^{2m-1}dt \right\}^\frac12 \right]\\
&\lesssim{\|\phi\|}_{\L^2(L^2;L^2)} T^\frac12\, \EE\left[  \left\{\sup_{t\in[0,T]} M(u^N(t))^{m-1} \right\}^\frac12 
    \left\{\sup_{t\in[0,T]} M(u^N(t))^{m} \right\}^\frac12 \right]\\
&\lesssim{\|\phi\|}_{\L^2(L^2;L^2)} T_0^\frac12 \left\{ \EE\left[ \sup_{t\in[0,T]} M(u^N(t))^{m-1} \right]\right\}^\frac12 
\left\{\EE\left[ \sup_{t\in[0,T]} M(u^N(t))^{m} \right]\right\}^\frac12
\end{align*}
Hence by Young's inequality, we infer that
\begin{align*}
\EE\left[\supT\eqref{mass-a1}\right]
\leq C_m{\|\phi\|}^2_{\L^2(L^2;L^2)} T_0\, \EE\left[ \sup_{t\in[0,T]} M(u^N(t))^{m-1} \right] 
 + \frac12\,\EE\left[ \sup_{t\in[0,T]} M(u^N(t))^{m} \right].
\end{align*}
In a straightforward way, we also have 
\begin{align*}
\EE\left[\supT\eqref{mass-a2}\right] 
&\leq m(m-1) {\|\phi\|}^2_{\L^2(L^2;L^2)} T_0\, \EE\left[ \sup_{t\in[0,T]} M(u^N(t))^{m-1} \right] ,\\
\EE\left[\supT\eqref{mass-a3}\right] 
&\leq 2m {\|\phi\|}^2_{\L^2(L^2;L^2)} T_0 \,\EE\left[ \sup_{t\in[0,T]} M(u^N(t))^{m-1} \right] .
\end{align*}
	Therefore, there is some $C_m>0$ such that 
	\begin{gather}
	\label{apriori-est-a-M1}
	\begin{split}
	\EE\left[\supT M(u^N(t))^m\right]&\le M(u_0)^m
	+ C_mT_0 \,\EE\left[\supT M(u^N(t))^{m-1}\right]\\ 
	&\qquad +\frac{1}{2}\,\EE\left[\supT M(u^N(t))^m\right] .
	\end{split}
	\end{gather}
We now wish to move the last term of \eqref{apriori-est-a-M1} to the left-hand side. However, we do not know a priori that the moments of $\supT M(u^N(t))$ are finite. To justify this, we note that \eqref{apriori-est-a-M1} holds with $T$ replaced by $T_R$, where 
$$T_R:= \sup\left\{t\in [0,T] : M(u^N(t))\leq R\right\},\quad R>0.$$
Now the terms that would be appearing in \eqref{apriori-est-a-M1} are finite 
and hence the formal manipulation is justified. 
Note that $T_R\to T$ almost surely as $R\to\infty$ because $u$ (and hence $u^N$) belongs in $C([0,T];H^s(\T^d))$ almost surely. Hence by letting $R\to\infty$ and invoking the monotone convergence theorem, one finds
\begin{equation}
\label{apriori-est-M-fin}
\EE\left[\supT M(u^N(t))^m\right]\le 2M(u_0)^m
	+ 2C_mT_0 \,\EE\left[\supT M(u^N(t))^{m-1}\right] .
\end{equation}
Hence, by induction on $m$, we obtain
	\begin{equation}\label{mass-est-a}
	\EE\left[\supT M(u^N(t))^m\right]\lesssim 1\,,
	\end{equation}
	where we note that the implicit constant is independent of $N$. 
	
	We now turn to estimating the energy. Applying It\^o's Lemma again, 
	we find that $E(u^N(t))^m$ equals 
	\begin{align}
	&E(u^N_0)^m\label{energy-a0}\\
	&\phantom{=}+ m\,\Im\brac{\sum_{|j|\leq N}\int_{0}^{t}{
			E(u^N(t'))^{m-1} \int_{\T^d} {|u^N|^{2k}u^N\phi^N e_j}{\,dx}
		}{\,d\beta_j(t')}} \label{energy-a1}\\
	&\phantom{=}-m\,\Im\brac{\sum_{|j|\leq N}\int_{0}^{t}{E(u^N(t'))^{m-1}
			 \int_{\T^d}\Delta \cj{u^N}\phi^Ne_j \,dx} \,d\beta_j(t')
		}\label{energy-a2}\\
	&\phantom{=}+(k+1)m\sum_{|j|\leq N}{\int_0^t E(u^N(t'))^{m-1}
			\int_{\T^d} {|u^N|^{2k}|\phi^N e_j|^2 \,dx\, dt' 
		}}\label{energy-a3}\\
	&\phantom{=}+m\norm{\nabla \phi^N}^2_{\L^2(L^2;L^2)}\intud{0}{t}{E(u^N(t'))^{m-1}}{t'}\label{energy-a4}\\
	&\phantom{=}+ \frac{m(m-1)}{2}\sum_{|j|\le N}\int_{0}^{t}{E(u^N(t'))^{m-2}
		\left|\int_{\T^d}{(-\Delta \cj{u^N}+ |u^N|^{2k}\cj{u^N})\phi e_jdx}\right|^2
	}{dt'}\label{energy-a5}.
\end{align}
\noi
We shall control here only the difficult term \eqref{energy-a1} 
as the other terms are bounded by similar lines of argument. 
	Firstly, by Burkholder-Davis-Gundy inequality (Lemma~\ref{BDG}), 
	we deduce
	\begin{align*}
	\EE\left[\supT\eqref{energy-a1}\right]
	&\le C_m \EE\left[\left\{\sum_{|j|\le N}\int_{0}^{T}
	E(u^N(t'))^{2(m-1)}\left|\int_{\T^d}{|u^N|^{2k}u^N\phi^N e_j}{\,dx}\right|^2
	{\,dt'}\right\}^\frac12\right] .
	\end{align*}
	Then, by duality and the (dual of the) Sobolev embedding $ H^1(\T^d) \hookrightarrow L^{2k+2}(\T^d)$, we have
\begin{align*}
\left|\int_{\T^d}{|u^N|^{2k}u^N\phi^N e_j}{\,dx}\right| &\leq 
  \left\| |u^N|^{2k}u^N\right\|_{H^{-1}(\T^d)} \|\phi^Ne_j\|_{H^1(\T^d)}\\
  &\lesssim \left\| |u^N|^{2k}u^N\right\|_{L^{\frac{2k+2}{2k+1}}(\T^d)} \|\phi e_j\|_{H^1(\T^d)}\\
  &\lesssim E(u^N)^{\frac{2k+1}{2k+2}} \|\phi e_j\|_{H^1(\T^d)}, 
\end{align*}
provided that $1+\frac1k \geq \frac{d}{2}$. 
Therefore, by  H\"older and Young inequalities, and similarly to the control of \eqref{mass-a1}, 
we have
	\begin{align*}
	\EE\left[\supT\eqref{energy-a1}\right] &\leq 
	C_m{\|\phi\|}^2_{\L^2(L^2;H^1)} T_0 \EE\left[ \sup_{t\in[0,T]} E(u^N(t))^{m-1} \right] 
 + \frac18\EE\left[ \sup_{t\in[0,T]} E(u^N(t))^{m-\frac{1}{2k+2}} \right]\\
 &\leq \tilde{C}_m{\|\phi\|}^2_{\L^2(L^2;H^1)} T_0 \EE\left[ \sup_{t\in[0,T]} E(u^N(t))^{m-1} \right] 
 + \frac18\EE\left[ \sup_{t\in[0,T]} E(u^N(t))^{m} \right],
	\end{align*}
where in the last step we used interpolation. 

We also have 
	\begin{align*}
	\EE\left[\supT\eqref{energy-a2}\right] &\le C_m\norm{\phi}_{\L^2(L^2;H^1)}\EE\left[\supT E(u^N(t))^{m-1}\right]+\frac18\EE\left[\supT E(u^N(t))^m\right]\\
	\EE\left[\supT\eqref{energy-a3}\right] &\le C_m\norm{\phi}_{\L^2(L^2;H^1)}^2+\frac18\EE\left[\supT E(u^N)^m\right] \\
	\EE\left[\supT\eqref{energy-a4}\right] &\le
	  C_m {\|\phi\|}_{\L^2(L^2;H^1)}^2 \EE\left[\supT E(u^N(t))^{m-1}\right] ,\\
	\EE\left[\supT\eqref{energy-a5}\right] &\le C\norm{\phi}^2_{\L^2(L^2;H^1)}
	 +\EE\left[\supT H(u^N(t))^{m-1}\right]+\frac18\EE\left[\supT H(u^N(t))^m\right].
	\end{align*}	
Gathering all the estimates, there exists $C_m>0$ such that 
	\begin{align*}
	\EE\left[\supT E(u^N(t))\right]&\le E(u_0)^m + C_{m} T_0\,\EE\left[\supT E(u^N(t))^{m-1}\right]
	+\frac{1}{2}\,\EE\left[\supT E(u^N(t))^m\right].
	\end{align*}
Similarly to passing from \eqref{apriori-est-a-M1} to \eqref{apriori-est-M-fin} and 
by induction on $m$, we deduce that
	\begin{equation}\label{energy-est-a}
	\EE\left[\supT E(u^N(t))^m\right]\lesssim 1 ,
	\end{equation}
with constant independent of $N$. 

\end{proof}
We now argue that the probabilistic a priori bounds in fact hold for solutions of the original SNLS.

\begin{corollary}\label{apriori-est-H1-L2}
For $u$ solution to \eqref{SNLS} with \eqref{add-noise}, 
the estimates 
\eqref{l2-est-a} and \eqref{h1-est-a} hold with $u$ in place of $u^N$ 
under the same assumptions as Proposition~\ref{apriori-est-M-a}. 
\end{corollary}
\begin{proof}
Let $\Lambda^N$ be the mild formulation of \eqref{SNLS_N}, more precisely, 
\begin{equation}\label{mild-N}
\Lambda^N(v):=S(t)u_0^N\pm i\int_0^tS(t-t')P_{\le N}\left(|v|^{2k}v\right)(t')\,dt'-i\int_0^tS(t-t')\phi^N\,dW^N(t')\,.
\end{equation}
Then $\Lambda^N$ is a contraction on a ball in $X^{1,\frac12-}([0,T])$ and has a unique fixed point $u^N$ that satisfies the bounds in Proposition \ref{apriori-est-M-a}. Hence it suffices to show that $u^N$ in fact converges to $u$ in $F_T:=L^2(\Omega;C([0,T]; H^s_x))$ for $s=0,1$. We only show $s=1$ since the proof of $s=0$ is the same. To this end, we consider the mild formulations of $u^N$ and $u$ and show that each piece of $u^N$ converges to the corresponding piece in $u$. Clearly, $S(t)u_0^N\to S(t)u_0$ in $F_T$. For the noise, let $\Psi^N(t)$ denote the stochastic convolution in \eqref{mild-N}. Then
\begin{align*}
\Psi(t)-\Psi^N(t) &=\left(
\sum_{|n|>N}\sum_{j\in\Z^d}+
\sum_{|n|\le N}\sum_{|j|>N}
\right) e_n\int_0^t e^{i(t-t')|n|^2} \ft{\phi e_j} (n) d\beta_j(t')\\
	&= \int_0^t S(t-t')P_{>N}\phi \,dW(t')+ \int_0^t S(t-t')\pi_{N} P_{\le N}\phi \,dW(t')\,,
\end{align*}
where $\pi_N$ denotes the projection onto the linear span of the orthonormal vectors $\{e_j: |j|>N\}$. By Lemma \ref{cts-stoc-conv-add}, the above is controlled by $$\norm{P_{>N}\circ\phi}_{\L^2(L^2;H^1)}^{2}+\norm{\pi_N P_{\le N} \phi}_{\L^2(L^2;H^1)}^{2}\,,$$ which tends to $0$ as $N\to\infty$ because both norms are tails of convergent series.

Finally we treat the nonlinear terms
\[Du(t):=\int_{0}^{t}{S(t-t')|u|^{2k}u(t')}{\,dt'}\quad\mbox{ and }\quad
D^{\le N}u(t):=\int_{0}^{t}{S(t-t')P_{\le N}\brac{|u|^{2k}u}(t')}{\,dt'}\,.\]
We first fix a path for which local well-posedness holds, and prove that $Du-D^{\le N}u\to 0$ in $X^{1,\frac12+}$. Firstly,
\begin{align*}
\norm{Du-D^{\le N}u}_{X^{1,\frac12+}([0,T])}
&\le\norm{\int_{0}^{t}{S(t-t')P_{\le N}(|u|^{2k}u-|u^N|^{2k}u^N)(t')}{\,dt'}}_{X^{1,\frac{1}{2}+}([0,T])}\\
&\phantom{= } \quad +
\norm{P_{>N}Du}_{X^{1,\frac{1}{2}+}([0,T])}
\end{align*}
By Lemmas \ref{lem:linests}, \ref{trilin} and \ref{multilin},  we have
\begin{align}
	\Irm &\les \left(\norm{u}_{X^{1,\frac12-}([0,T])}^{2k}
	  +\norm{u^N}_{X^{1,\frac12-}([0,T])}^{2k}\right)\norm{u-u^N}_{X^{1,\frac12-}([0,T])}\label{I-est1}\\
	\IIrm &\les \norm{u}_{X^{1,\frac12-}([0,T])}^{2k+1}\label{II-est1}	
\end{align} 
In particular, \eqref{II-est1} implies $Du\in X^{1,\frac12+}([0,T])$, and hence $\IIrm\to 0$ as $N\to\infty$. We claim that $\Irm\to 0$ as $N\to\infty$ as well. 
Indeed, $\Lambda^N$ and $\Lambda$ are contractions with fixed points $u^N$ and $u$ respectively, hence
\begin{align*}
\norm{u-u^N}_{X^{1,\frac12-}([0,T])}
&\le \norm{\Lambda(u)-\Lambda^N(u)}_{X^{1,\frac12-}([0,T])}+
 \norm{\Lambda^N(u)-\Lambda^N(u^N)}_{X^{1,\frac12-}([0,T])}\\
	&\le \norm{\Lambda(u)-\Lambda^N(u)}_{X^{1,\frac12-}([0,T])}+
	\frac12 \norm{u-u^N}_{X^{1,\frac12-}([0,T])}\,.
\end{align*}
By rearranging, it suffices to show that the first term on the right-hand side above tends to $0$ as $N\to\infty$. Now
\begin{align*}
\norm{\Lambda(u)-\Lambda^N(u)}_{X^{1,\frac12-}([0,T])}
	&\le \norm{P_{>N}S(t)u_0}_{X^{1,\frac12-}([0,T])}\\
	&\quad + \norm{P_{>N}\int_0^tS(t-t')|u|^{2k}u(t')\,dt'}_{X^{1,\frac12-}([0,T])}\\
&\quad +\norm{\Psi^{>N}}_{X^{1,\frac12-}([0,T])}\,.
\end{align*}
By similar arguments as above, all the terms on the right go to $0$ as $N\to\infty$. This proves our claim. By the embedding $X^{1,\frac{1}{2}+}([0,T])\subset C([0,T];H^1(\T^d))$, we have that
\begin{equation}\label{nonlinear-convergence}\left\|Du-D^{\le N}u\right\|_{C([0,T];H^1)}\to 0
\end{equation}
almost surely as $N\to\infty$. By the dominated convergence theorem, we have $Du-D^{\le N}u\to 0$ in $F_T$. This concludes our proof.
\end{proof}
Finally, we conclude the proof of global well-posedness for the additive case.
\begin{proof}[Proof of Theorem \ref{gwp:add}]
Let $s\in\{0,1\}$ be the regularity of $u_0$ from Theorem \ref{gwp:add}. Let $\eps>0$ and $T>0$ be given. We claim that there exists an event $\Omega_\eps$ such that a solution 
$u\in X^{s,b}([0,T])\cap C([0, T]; H^s(\T^d))$ exists on $[0,T]$ in $\Omega_\eps$ 
and $\P(\Omega\setminus \Omega_\eps)<\eps$. If this claim holds, then by setting
\[\Omega^*=\bigcup_{n=1}^\infty\Omega_\frac{1}{n},\]
we have that $\P(\Omega^*)=1$ and $u$ exists on $[0,T]$, proving the theorem. Let $\delta\in(0,1)$ be a small quantity chosen later. We subdivide $[0,T]$ into $M=\left\lceil\frac{T}{\delta}\right\rceil$ subintervals $I_k=[(k-1)\delta,k\delta]$. Let
\[\Omega_0=\bigcap_{k=1}^M\left\{\omega\in\Omega:\xsbt{\intud{(k-1)\delta}{t}{S(t-t')\phi}{W(t')}}{s}{b}{I_k}\le L\right\},\]
where $L>0$ is some large quantity determined later. Now by Chebyshev's inequality and Lemma \ref{stoc-conv-est-add},
\begin{align*}
\P(\Omega\setminus\Omega_0)
&=\sum_{k=1}^M\P\brac{\xsbt{\intud{(k-1)\delta}{t}{S(t-t')\phi}{W(t')}}{s}{b}{I_k}>L}\\
&\le \sum_{k=1}^M\frac{1}{L^2}\EE\left[\xsbt{\intud{(k-1)\delta}{t}{S(t-t')\phi}{W(t')}}{s}{b}{I_k}^2\right]\\
&\lesssim \sum_{k=1}^M\frac{\delta(\delta^2+1)}{L^2}\norm{\phi}_{\L^2(L^2;L^2)}^2\\
&\le \frac{2M\delta}{L^2}\norm{\phi}_{\L^2(L^2;L^2)}^2\\
&\lesssim \frac{T}{L^2}\norm{\phi}_{\L^2(L^2;L^2)}^2\,.
\end{align*}
By choosing $L=L(\eps,T,\phi)$ sufficiently large, we may therefore bound $\P(\Omega^c_0)$ above by $\frac{\eps}{2}$. Now let 
\[ R=\max\left\{\norm{u_0}_{H^s}, L\right\}\,.\]
By local theory, there exists a unique solution $u(t)$ to \eqref{SNLS} with time of existence $T_\text{max}$ given in \eqref{time-of-existence-a}. In particular, we note that for $\omega\in\Omega_0$,
\begin{equation}\label{toe-ineq}
c\Big({\norm{u_0}_{H^s}+\norm{\Psi}_{X^{s,b}_{[0,\delta]}}}\Big)^{-\theta}
	\ge c\big({R+L}\big)^{-\theta},
\end{equation}
where $c$ is as in \eqref{time-of-existence-a}. 
By choosing $\delta=\delta(R,L):=c ({R+L})^{-\theta}$, 
we see that $u(t)$ exists for $t\in[0, \delta]$ for all $\omega\in\Omega_0$. Now define
\[\Omega_1=\left\{\omega\in\Omega_0:\norm{u(\delta)}_{H^s}\le R\right\}\,.\]
By the same argument, $u(t)$ exists for $t\in(\delta,2\delta)$ for all $\omega\in\Omega_1$. Iterating this argument, we have a chain of events $\Omega_0\supseteq \Omega_1 \supseteq \dots \supseteq  \Omega_{M-1}$ where
\[\Omega_k=\left\{\omega\in\Omega_{k-1}:\norm{u(k\delta)}_{H^s}\le R\right\} \]
and $u(t)$ exists for all $t\in[0,(k+1)\delta]$ on $\Omega_k$. Setting $\Omega_\eps:=\Omega_{M-1}$, $u(t)$ exists on the full interval $[0,T]$ on $\Omega_\eps$. It remains to check that $\Omega\setminus\Omega_\eps$ remains small. 
By Corollary \ref{apriori-est-H1-L2}, we have
\begin{align*}
\P(\Omega_\eps)&\le\P(\Omega\setminus\Omega_0)+\sum_{k=0}^{M-1}\P(\Omega_{k+1}^c\cap\Omega_k)\\
&\le\frac{\eps}{2}+\sum_{k=0}^{M-1}\P\brac{\left\{\norm{u((k+1)\delta)}_{H^s}>R\right\}\cap\Omega_k}\\
&\le\frac{\eps}{2}+\sum_{k=0}^{M-1}\frac{1}{R^p}\EE\left[\one_{\Omega_k}\norm{u((k+1)\delta)}_{H^s}^p\right]\\
&\le\frac{\eps}{2}+\frac{MC_1}{R^p}\\
&\le \frac{\eps}{2}+\frac{2TC_1(R+L)^{\theta}}{cR^p}\,,
\end{align*}
for any $p\in\NN$. We further enlarge $R$ if necessary by setting
\[R=\max\left\{\frac{2TC_1}{c}+1, L, \norm{u_0}_{H^s}\right\}\,,\]
where have that
\[\P(\Omega_\eps)\le \frac{\eps}{2}+2^\theta R^{\theta-p+1}\,.\]
This is smaller than $\eps$ provided we choose $p=p(\eps,\theta)>0$ sufficiently large. Thus $\Omega_\eps$ satisfies our claim.
\end{proof}

\subsection{SNLS with multiplicative noise}
In order to globalize solutions of SNLS, for the multiplicative noise case, 
we need to prove probabilistic control  of the $X^{s,b}$-norm of the solutions of the truncated SNLS uniformly in the truncation parameter (Lemma~\ref{finite-xsb}). This requires a priori bounds on mass and energy of solutions.  

From Subsection \ref{sect:LWPm}, we obtained a local solution of the multiplicative \eqref{SNLS} with time of existence
\[\tau^*=\lim_{R\to\infty}\tau\sub{R}\,.\]
Under the hypotheses of Theorem~\ref{gwp:mult}, 
we shall prove global well-posedness by showing that $\tau^*=\infty$ almost surely.
\begin{proposition}\label{apriori-est-M-m}
Let $T_0>0$ and $\phi$ be as in Theorem~\ref{gwp:mult}. 
 Suppose that 
	 $u(t)$ is a solution for \eqref{SNLS} with $F(u,\phi\xi)=u\cdot\phi\xi$ on $t\in[0,T]$  
	for some stopping time $T\in[0,T_0\wedge \tau^*)$. Let $C(\phi)$ be as in \eqref{defn:Cphi}. Then for any $m\in\NN$, there exists 
	$C_1=C_1(m,M(u_0), T_0, C(\phi))>0$ such that
	\begin{align}
	\EE\left[\sup_{0\le t\le T} M(u(t))^m \right] 
	\le C_1\,.
	\end{align}
	Furthermore, if\eqref{SNLS} is defocusing,  
	there exists $C_2=C_2(m,E(u_0), T_0,C(\phi))>0$ such that	
	\begin{align}
	\label{h1-est-m}
	\EE\left[\sup_{0\le t\le T} E(u(t))^m \right] 
	&\le C_2\,.
	\end{align}
\end{proposition}
\begin{proof}
We consider the 
frequency truncated equation \eqref{SNLS_N} and apply It\^o's Lemma to obtain
\begin{align}
M(u^N(t))^m&=M(u_0^N)^m\notag\\
&\phantom{=}+
m\,\Im\brac{\sum_{|j|\le N}\int_{0}^{t}{M(u^N(t'))^{m-1}\int_{\T^d}|u^N(t')|^2\phi^Ne_j\,dx}{\,d\beta_j(t')}}\label{mass-m1}\\
&\phantom{=}+m(m-1)\sum_{|j|\le N}\int_{0}^{t}{M(u^N(t'))^{m-2}\left|\int_{\T^d}|u^N(t')|^2\phi^Ne_j\,dx\right|^2}{\,dt'}\label{mass-m2}\\
&\phantom{=}+m(m-1)\sum_{|j|\le N}\int_0^t{M(u^N(t'))^{m-1} \int_{\T^d}|u(t')\phi e_j|^2\,dx}{\,dt'}\label{mass-m3}\,.
\end{align}
To bound \eqref{mass-m1}, we use Burkholder-Davis-Gundy inequality (Lemma~\ref{BDG}) 
and use a similar argument as in the proof of Lemma \ref{stoc-conv-est-mult} to get
\begin{align*}
\EE\left[\supT\eqref{mass-m1}\right]
	&\les \EE\left[\sum_{|j|\le N}
	\left(
	\int_{0}^{T}M(u^N(t'))^{2(m-1)}\left|\int_{\T^d}|u^N(t')\phi e_j|^2\,dx\right|^2\,dt'
	\right)^\frac12\right]\\
&\le C(\phi)^2\EE\left[\left(\int_0^TM(u^N(t'))^{2m}\right)^\frac12\right]\\
	&\le C(\phi)^2 \brac{\EE\left[\supT M(u^N(t))^m\right]}^\frac{1}{2}\brac{\EE\left[\int_{0}^{T}{{M(u^N(t'))^{m}}}{\,dt'}\right]}^\frac{1}{2}
\end{align*}
Similarly, one obtains
\begin{align*}
\EE\left[\supT\left\{\eqref{mass-m2}+\eqref{mass-m3}\right\}\right]
  &\les C(\phi)\EE\left[\int_0^TM(u^N(t'))^m\,dt'\right]
\end{align*}
Hence there is a constant $C_1=C_1(m, M(u_0), T, C(\phi))$ such that
\begin{align*}
\EE\left[\supT M(u^N(t))^m\right]
	&\le C_1 +C_1\,\EE\left[\int_0^TM(u^N(t'))^m\,dt'\right]\\
&\phantom{=}+C(\phi)^2 \brac{\EE\left[\supT M(u^N(t))^m\right]}^\frac{1}{2}\brac{\EE\left[\int_{0}^{T}{{M(u^N(t'))^{m}}}{\,dt'}\right]}^\frac{1}{2}
\end{align*}
The left-hand side is bounded above by $3\mathcal{M}$, 
where $\mathcal{M}$ is maximum of the three terms of the right-hand side. 
In any of the three cases, 
we may conclude the proof via simple rearrangement arguments and Gronwall's inequality.

Turning to the energy, we use It\^o's Lemma and the defocusing equation to obtain 
that $E(u^N(t))^m$ equals 
\begin{align}
&E(u^N_0)^m\label{energy-m0}\\
&+ m\,\Im\brac{\sum_{|j|\le N}\int_{0}^{t}{
		E(u^N(t'))^{m-1}\int_{\T^d}{|u^N|^{2(k+1)}\phi^N e_j}{\,dx}
	}{\,d\beta_j(t')}}\label{energy-m1}\\
&-m\,\Im\brac{\sum_{|j|\le N}\int_{0}^{t}{E(u^N(t'))^{m-1}
		\int_{\T^d}{(\Delta\cj{u^N}){u^N\phi^N e_j}}\,dx
	}{\,d\beta_j(t')}}\label{energy-m2}\\
&+ m(k+1)\sum_{|j|\le N}{\int_{0}^{t}{E(u^N(t'))^{m-1}
		\int_{\T^d}{|u^N|^{2(k+1)}|\phi^N e_j|^2}{\,dx}
	}{\,dt'}}\label{energy-m3}\\
&+m\sum_{|j|\le N}\int_{0}^{t}{E(u^N(t'))^{m-1}\int_{\T^d}|\nabla{(u^N\phi^N e_j)}(n)|^2}\,dx{\,dt'}\label{energy-m4}\\
&+\frac{m(m-1)}{2}\brac{\sum_{|j|\le N}\int_{0}^{t}{E(u^N(t'))^{m-2}\left|\int_{\T^d}{\brac{-u^N\Delta\cj{u^N}+|u^N|^{2k+1}}\phi^N e_j}{\,dx}\right|^2 }{\,dt'} }\label{energy-m5}
\end{align}
For \eqref{energy-m1}, we use Burkholder-Davis-Gundy inequality (Lemma~\ref{BDG}) to get
\begin{align*}
\EE\left[\supT\eqref{energy-m1}\right]
	&\les\EE\left[\left(\sum_{|j|\le N}\int_0^TE(u^N(t'))^{2(m-1)}
	\left|
	\int_{\T^d}|u^N|^{2k+2}\phi^Ne_j\,dx
	\right|^2\,dt'
	\right)^\frac12\right]\,.
\end{align*}
Now, with $r$ as in Theorem~\ref{lwp:mult}, 
\begin{align*}
\left|
\int_{\T^d}|u^N|^{2k+2}\phi^Ne_j\,dx\right|^2
	&\le\norm{u^N}_{L^{2k+2}_x}^{2(2k+2)}\norm{\phi^Ne_j}_{L^\infty_x}^2
	\le E(u)^{2}\norm{\ft{\phi^Ne_j}}_{\ell^1}^2\\
	& \les E(u)^{2} \|\phi^N e_j\|_{\F L^{s,r}} \,,
\end{align*}
where for the last step see Lemma~\ref{lem3p5}
Therefore, by H\"older's inequality and \eqref{defn:Cphi}, 
\begin{align*}
\EE\left[\supT\eqref{energy-m1}\right]
&\les C(\phi)\,  \EE\left[\left(\int_0^TE((u^N(t')))^{2m}\,dt'\right)^\frac12\right]\\
	&\le C(\phi) \brac{\EE\left[\supT E(u^N(t))^m\right]}^\frac{1}{2}\brac{\EE\left[\int_{0}^{T}{{E(u^N(t'))^{m}}}{\,dt'}\right]}^\frac{1}{2}\,.
\end{align*}
Similarly, we bound the other terms as follows:
\begin{align*}
&\EE\left[\supT\eqref{energy-m2}\right]\les C(\phi) \brac{\EE\left[\supT E(u^N(t))^m\right]}^\frac{1}{2}\brac{\EE\left[\int_{0}^{T}{{E(u^N(t'))^{m}}}{\,dt'}\right]}^\frac{1}{2}\\
&\EE\left[\supT\{\eqref{energy-m3}+\eqref{energy-m4}+\eqref{energy-m5}\}\right]
\les 
C(\phi)^2 \EE\left[\int_0^TE(u^N(t'))^m\,dt'\right]
\end{align*}
It follows that there is a constant $C_2=C_2(m, E(u_0), T, C(\phi))$ such that
\begin{align*}
\EE\left[\supT E(u^N(t))^m\right]
&\le C_2 +C_2\,\EE\left[\int_0^TE(u^N(t'))^m\,dt'\right]\\
&\phantom{=}+C_2 \brac{\EE\left[\supT E(u^N(t))^m\right]}^\frac{1}{2}\brac{\EE\left[\int_{0}^{T}{{E(u^N(t'))^{m}}}{\,dt'}\right]}^\frac{1}{2}\,.
\end{align*}
Arguing in the same way as for the mass of $u^N$ yields the estimate for the energy of $u^N$. This proves the proposition for $u^N$ in place of $u$. The proposition then follows by letting $N\to\infty$. 
\end{proof}
We now prove the following probabilistic a priori bound on the $X^{s,b}$-norm of a solution.
\begin{lemma}\label{finite-xsb}
Let $T, R>0$. Let $u\sub{R}$ be the unique solution of \eqref{SNLSm-R} on $[0,T]$. 
There exists $C_1=C_1(\norm{u_0}_{L^2}, T, C(\phi))$ such that
\[\EE\left[\xsbt{u\sub{R}}{0}{b}{[0,T\wedge\tau_R]}\right]\le C_1\,.\]
Moreover, if \eqref{SNLSm-R} is defocusing,   
there also exists $C_2=C_2(\norm{u_0}_{H^1}, T,C(\phi))$ such that
\[\EE\left[\xsbt{u\sub{R}}{1}{b}{[0,T\wedge\tau_R]}\right]\le C_2\,.\]
The constants $C_1$ and $C_2$ are independent of $R$.
\end{lemma}
\begin{proof}
Let $\tau$ be a stopping time so that $0< \tau\le T\wedge\tau_R$. By a similar argument used in local theory, we have
\begin{equation}\label{some-xsb-bound}
\begin{aligned}
\xsbt{u\sub{R}}{s}{b}{[0,\tau]}&\le C_1\norm{u\sub{R}(0)}_{H^s}+C_2\tau^\delta\xsbt{u\sub{R}}{s}{b}{[0,\tau]}^{2k+1}+\xsbt{\Psi}{s}{b}{[0,\tau]}\\
&\le C_1\norm{u\sub{R}}_{C([T\wedge\tau_R];H^s)}+C_2\tau^\delta\xsbt{u\sub{R}}{s}{b}{[0,\tau]}^{2k+1}+\xsbt{\Psi}{s}{b}{[0,T\wedge\tau_R]}\,.
\end{aligned}
\end{equation}
Let $K=C_1\norm{u\sub{R}}_{C([T\wedge\tau_R];H^s)}+\xsbt{\Psi(t)}{s}{b}{[0,T\wedge\tau_R]}$. We claim that if $\tau\sim K^{-\frac{2k}{\delta}}$, then
\begin{align}\label{continuity1}
\xsbt{u\sub{R}}{s}{b}{[0,\tau]}\lesssim K\,.
\end{align}
To see this, we note that the polynomial
\begin{equation}\label{poly}
p\sub{\tau}(x)=C_2\tau^\delta x^{2k+1}-x+K
\end{equation}
has exactly one positive turning point at
\[x'_+=\brac{{(2k+1)C_2\tau^\delta}}^{-\frac{1}{2k}}\,\]
and that $p_\tau(x'_+)<0$ if we choose $\tau=cK^{-\frac{2k}{\delta}}$. For this choice, we have $p_\tau(0)=K>0$ and hence $p_\tau(x)>0$ for $0\le x< x_+$ where $x_+$ is the unique positive root below $x_+'$. Now \eqref{some-xsb-bound} is equivalent to $p_\tau\big(\norm{u_R}_{X^{s,b}([0,\tau])}\big)\ge 0$. 
But since $g(\,\cdot\,):=\norm{u_R}_{X^{s,b}([0,\,\cdot\,])}$ is continuous and $g(0)=0$, we must have
\[g(\tau)<x_+'\sim \tau^{-\frac{\delta}{2k}}\sim K\,,\]
which proves \eqref{continuity1}. Iterating this argument, we find that 
\begin{align}\label{continuityk}
\norm{u\sub{R}}_{X^{s,b}([(j-1)\tau,j\tau])}\lesssim \norm{u\sub{R}}_{C([0,T\wedge\tau_R];H^s)}+\xsbt{\Psi(t)}{s}{b}{[0,T\wedge\tau_R]}\,
\end{align}
for all integer $1\le j\le \lceil\frac{T\wedge\tau_R}{\tau}\rceil=:M$. Putting everything together, we have
\begin{align*}
\norm{u\sub{R}}_{X^{s,b}([0,T\wedge\tau_R])}
&\le\sum_{j=1}^{M}\norm{u\sub{R}}_{X^{s,b}([(j-1)\tau,j\tau])}\\
&\lesssim \frac{T\wedge\tau_R}{\tau}\brac{\norm{u\sub{R}}_{C([0,T\wedge\tau_R];H^s)}+\xsbt{\Psi}{s}{b}{[0,T\wedge\tau_R]} }\\
&\lesssim T\brac{\norm{u\sub{R}}_{C([0,T\wedge\tau_R];H^s)}+\xsbt{\Psi}{s}{b}{[0,T\wedge\tau_R]} }^{\frac{2k}{\delta}+1}\,.
\end{align*}
By Proposition \ref{apriori-est-M-m} and Lemma \ref{stoc-conv-est-mult}, all moments of the last two terms above are finite. This proves Lemma~\ref{finite-xsb}. 
\end{proof}

We can now conclude the proof of Theorem \ref{gwp:mult}.

\begin{proof}[Proof of Theorem \ref{gwp:mult}]
Fix $T>0$. Since $\tau\sub{R}$ is increasing in $R$,
\begin{align*}
\P(\tau^*<T) &=\lim_{R\to\infty}\P(\tau_R<T)= \lim_{R\to\infty}\P\left(\norm{u_R}_{X^{s,b}([0,T\wedge\tau_R])}\ge R\right)\\
&\le \lim_{R\to\infty}\frac1R\EE\left[\norm{u_R}_{X^{s,b}([0,T\wedge \tau_R])}\right]\,.
\end{align*}
But then the right-hand side equals $0$ by Lemma \ref{finite-xsb}. 
It follows that $\tau^*=\infty$ almost surely.
\end{proof}

\end{document}